\documentclass{amsart}
\usepackage{amsmath}
\usepackage{amscd}
\usepackage{amssymb}
\usepackage{amsthm}
\usepackage{color}
\usepackage{tcolorbox}
\RequirePackage{filecontents}
\usepackage{fullpage}
\usepackage{longtable}
\usepackage{pdflscape}
\usepackage{afterpage}
\usepackage[numbers]{natbib}  
\usepackage[draft]{hyperref}
\theoremstyle{plain}
\newtheorem{prop}{Proposition}
\newtheorem{lem}[prop]{Lemma}
\newtheorem{thrm}[prop]{Theorem}

\theoremstyle{definition}
\newtheorem{defn}[prop]{Definition}
\newtheorem{rem}[prop]{Remark}

\newenvironment{psmallmatrix}
  {\left(\begin{smallmatrix}}
  {\end{smallmatrix}\right)}

\makeatletter
\newcommand{\addresseshere}{%
  \enddoc@text\let\enddoc@text\relax
}
\makeatother

\author{Brandon Williams }

\address{Fachbereich Mathematik \\ Technische Universit\"at Darmstadt \\ 64289 Darmstadt, Germany}

\email{bwilliams@mathematik.tu-darmstadt.de}

\thanks{This research is supported by a postdoctoral fellowship of the LOEWE research unit Uniformized Structures in Arithmetic and Geometry.}

\begin{document}

\nocite{*}

\title{Two graded rings of Hermitian modular forms}

\begin{abstract} We give generators and relations for the graded rings of Hermitian modular forms of degree two over the rings of integers in $\mathbb{Q}(\sqrt{-7})$ and $\mathbb{Q}(\sqrt{-11})$. In both cases we prove that the subrings of symmetric modular forms are generated by Maass lifts. The computation uses a reduction process against Borcherds products which also leads to a dimension formula for the spaces of modular forms.
\end{abstract}

\maketitle

\section{Introduction}

Hermitian modular forms of degree $n \in \mathbb{N}$ are modular forms that transform under an action of the split-unitary group $\mathrm{SU}(n,n;\mathcal{O})$ with entries in some order $\mathcal{O}$ in an imaginary-quadratic number field. Through the natural embedding of $\mathrm{SU}(n,n;\mathcal{O})$ in $\mathrm{Sp}_{4n}(\mathbb{Z})$, the Shimura variety attached to $\mathrm{SU}(n,n;\mathcal{O})$ parameterizes certain principally polarized $(2n)$-dimensional abelian varieties, namely the abelian varieties $A$ of \emph{Weil type}, i.e. admitting multiplication by $\mathcal{O}$ in such a way that the eigenvalues of $\mathcal{O}$ acting on $A$ occur in complex-conjugate pairs. (These were investigated by Weil in connection with the Hodge conjecture; see for example the discussion in \cite{LPS}, which also explains the connection to orthogonal Shimura varieties when $n=2$.) To study such objects it is helpful to have coordinates on the moduli space; in other words, generators for graded rings of Hermitian modular forms.

In \cite{DK1}, \cite{DK2}, Dern and Krieg began a program to compute these rings in degree $n=2$ based on Borcherds' \cite{B} theory of orthogonal modular forms with Heegner divisors (and the exceptional isogeny from $\mathrm{SU}(2,2)$ to $\mathrm{SO}(2,4)$). In particular they give an explicit description of the modular fourfolds associated to $\mathrm{SU}(2,2,\mathcal{O})$ where $\mathcal{O}$ is the maximal order in $\mathbb{Q}(\sqrt{-3})$ and $\mathbb{Q}(\sqrt{-1})$ (where the fourfold is rational) and in $\mathbb{Q}(\sqrt{-2})$ (where it is not). The contribution of this note is to carry out these computations for the imaginary-quadratic fields of the smallest two remaining discriminants: $\mathbb{Q}(\sqrt{-7})$ and $\mathbb{Q}(\sqrt{-11})$.

The rough idea of \cite{DK1}, \cite{DK2} is similar to the well-known computation of the ring of elliptic modular forms, $M_{\ast}(\mathrm{SL}_2(\mathbb{Z})) = \mathbb{C}[E_4,E_6]$. The Riemann-Roch theorem (in the form of the ``$k/12$ formula'') shows that every modular form of weight not divisible by $6$ has a zero at the elliptic point $\rho = e^{2\pi i / 3}$, and that the Eisenstein series $E_4$ and $E_6$ have no zeros besides a simple zero at $\rho$ and at $i$ (and their conjugates under $\mathrm{SL}_2(\mathbb{Z})$), respectively. Now every form in $M_{\ast}(\mathrm{SL}_2(\mathbb{Z}))$ of weight not a multiple of $6$ is divisible by $E_4$, and every form of weight $6k$ becomes divisible by $E_4$ after subtracting some scalar multiple of $E_6^k$. The claim follows by induction on the weight, together with the fact that modular forms of weight $k \le 0$ are constant.

In the $\mathrm{SU}(2,2)$ case the role of $E_4$ above is played by a Borcherds product; the elliptic point $\rho$ is replaced by the Heegner divisors; and the evaluation at $\rho$ is replaced by the pullbacks, which send Hermitian modular forms to Siegel paramodular forms of degree two. With increasing dimension and level, the Heegner divisors which occur as divisors of modular forms are more complicated and the pullback maps to Heegner divisors are rarely surjective. To overcome these issues our basic argument is as follows. We construct Hermitian modular forms (Eisenstein series, theta lifts, pullbacks from $\mathrm{O}(2,5)$, theta series, etc; here, theta lifts and Borcherds products turn out to be sufficient) and compute their pullbacks to paramodular forms. At the same time we use the geometry of the Hermitian modular fourfold (in particular the intersections of special divisors) to constrain the images of the pullback maps, with the goal of determining sufficiently many images completely. There seems to be no reason in general to believe that this procedure will succeed, and as the discriminant of the underlying field increases it certainly becomes more difficult; however, when this computation does succeed it is straightforward to determine the complete ring structure.

This note is organized as follows. In section 2 we review Hermitian and orthogonal modular forms, theta lifts and pullbacks. In section 3 we recall the structure of the graded rings of paramodular forms of degree two and levels $1,2,3$. In sections 4 and 5 we compute the graded rings of Hermitian modular forms for the rings of integers of $\mathbb{Q}(\sqrt{-7})$ and $\mathbb{Q}(\sqrt{-11})$ by reducing against distinguished Borcherds products of weight $7$ and $5$, respectively. (The ideal of relations for $\mathbb{Q}(\sqrt{-11})$ is complicated and left to an auxiliary file.) In section 6 we compute the dimensions of spaces of Hermitian modular forms.

\textbf{Acknowledgments.} I am grateful to Jan H. Bruinier, Aloys Krieg and John Voight for helpful discussions.

\section{Preliminaries}

In this section we review some facts about Hermitian modular forms of degree two and the related orthogonal modular forms. For a more thorough introduction the book \cite{K} and the dissertation \cite{Dern} are useful references.

\subsection{Hermitian modular forms of degree two}
Let $\mathbf{H}_2$ denote the Hermitian upper half-space of degree two: the set of complex $(2 \times 2)$-matrices $\tau$ for which, after writing $\tau = x + iy$ where $x = \overline{x}^T$ and $y = \overline{y}^T$, the matrix $y$ is positive-definite. The split-unitary group $$\mathrm{SU}_{2,2}(\mathbb{C}) = \Big\{ M \in \mathrm{SL}_4(\mathbb{C}): \; M^T J \overline{M} = J \Big\}, \; \; J = \begin{psmallmatrix} 0 & 0 & -1 & 0 \\ 0 & 0 & 0 & -1 \\ 1 & 0 & 0 & 0 \\ 0 & 1 & 0 & 0 \end{psmallmatrix}$$ acts on $\mathbf{H}_2$ by M\"obius transformations: $$M \cdot \tau = (a \tau + b)(c \tau + d)^{-1}, \; \; M = \begin{psmallmatrix} a & b \\ c & d \end{psmallmatrix} \in \mathrm{SU}_{2,2}(\mathbb{C}), \; \tau \in \mathbf{H}_2.$$

Fix an order $\mathcal{O}$ in an imaginary-quadratic number field $K$. A \textbf{Hermitian modular form} of weight $k \in \mathbb{N}_0$ (and degree two) is a holomorphic function $F : \mathbf{H}_2 \rightarrow \mathbb{C}$ which satisfies $$F(M \cdot \tau) = \mathrm{det}(c \tau + d)^k F(\tau) \; \text{for all} \; M = \begin{psmallmatrix} a & b \\ c & d \end{psmallmatrix} \in \mathrm{SU}_{2,2}(\mathcal{O}) \; \text{and} \; \tau \in \mathbf{H}_2.$$ Note that $F$ extends holomorphically to the Baily-Borel boundary (i.e. Koecher's principle) as this contains only components of dimension $1$ and $0$. \textbf{Cusp forms} of weight $k$ are modular forms which tend to zero at each one-dimensional cusp: that is, modular forms $f$ for which $$\lim_{y \rightarrow \infty} \Big( f \Big|_k M \Big)(iy) = 0 \; \text{for all} \; M \in \mathrm{SU}_{2,2}(K).$$

\subsection{Orthogonal modular forms and Hermitian modular forms} Suppose $\Lambda = (\Lambda,Q)$ is an $\ell$-dimensional positive-definite even lattice; that is, $\Lambda$ is a free $\mathbb{Z}$-module of rank $\ell$ and $Q$ is a positive-definite quadratic form on $\Lambda \otimes \mathbb{R}$ taking integral values on $\Lambda$. One can define an upper half-space $$\mathbb{H}_{\Lambda} = \{(\tau,z,w): \; \tau,w \in \mathbb{H}, \; z \in \Lambda \otimes \mathbb{C} , \; Q(\mathrm{im}(z)) < \mathrm{im}(\tau) \cdot \mathrm{im}(w)\} \subseteq \mathbb{C}^{\ell + 2}.$$ This is acted upon by $\mathrm{SO}^+(\Lambda \oplus \mathrm{II}_{2,2})$ (the connected component of the identity) by M\"obius transformations. To make this explicit it is helpful to fix a Gram matrix $\mathbf{S}$ for $Q$ and realize $\mathrm{SO}^+(\Lambda \oplus II_{2,2})$ as a subgroup of those matrices which preserve the block matrix $\begin{psmallmatrix} 0 & 0 & 0 & 0 & 1 \\ 0 & 0 & 0 & 1 & 0 \\ 0 & 0 & \mathbf{S} & 0 & 0 \\ 0 & 1 & 0 & 0 & 0 \\ 1 & 0 & 0 & 0 & 0 \end{psmallmatrix} \in \mathbb{Z}^{6 \times 6}$ under conjugation. For such a matrix $M$ and $(\tau,z,w) \in \mathbb{H}_{\Lambda}$, one can define $M \cdot (\tau,z,w) = (\tilde \tau, \tilde z, \tilde w) \in \mathbb{H}_{\Lambda}$ by $$M \begin{psmallmatrix} Q(z) - \tau w \\ \tau \\ z \\ w \\ 1 \end{psmallmatrix} = j(M;\tau,z,w) \begin{psmallmatrix} Q(\tilde z) - \tilde \tau \tilde w \\ \tilde \tau \\ \tilde z \\ \tilde w \\ 1 \end{psmallmatrix} \; \text{for some} \; j(M;\tau,z,w) \in \mathbb{C}^{\times}.$$

The \textbf{orthogonal modular group} $\Gamma_{\Lambda}$ is the discriminant kernel of $\Lambda \oplus \mathrm{II}_{2,2}$; that is, the subgroup of $\mathrm{SO}^+(\Lambda \oplus \mathrm{II}_{2,2})$ which acts trivially on $\Lambda'/\Lambda$. An \textbf{orthogonal modular form} is then a holomorphic function $f : \mathbb{H}_{\Lambda} \rightarrow \mathbb{C}$ which satisfies $$f(M \cdot (\tau,z,w)) = j(M;\tau,z,w)^k f(\tau,z,w)$$ for all $M \in \Gamma_{\Lambda}$ and $(\tau,z,w) \in \mathbb{H}_{\Lambda}$. (There is again a boundedness condition at cusps which is automatic by Koecher's principle.)

Hermitian modular forms for $\mathrm{SU}_{2,2}(\mathcal{O}_K)$ are more or less the same as orthogonal modular forms for the lattice of integers $(\Lambda,Q) = (\mathcal{O}_K,N_{K/\mathbb{Q}})$ of $K$. One way to see this is as follows. The complex space of antisymmetric $(4 \times 4)$-matrices admits a nondegenerate quadratic form $\mathrm{pf}$ (the \emph{Pfaffian}, a square root of the determinant) which is preserved under the conjugation action $M \cdot X = M^T X M$ by $\mathrm{SL}_4(\mathbb{C})$; explicitly, $$\mathrm{pf} \begin{psmallmatrix} 0 & a & b & c \\ -a & 0 & d & e \\ -b & -d & 0 & f \\ -c & -e & -f & 0 \end{psmallmatrix} = af - be + cd.$$ The conjugation action identifies $\mathrm{SL}_4(\mathbb{C})$ with the spin group $\mathrm{Spin}(\mathrm{pf}) = \mathrm{Spin}_6(\mathbb{C})$. The six-dimensional real subspace $$V =  \Big\{ \begin{psmallmatrix} 0 & a & b & c \\ -a & 0 & d & -\overline{b} &  \\ -b & -d & 0 & f \\ -c & \overline{b} & -f & 0 \end{psmallmatrix}: \; a,c,d,f \in \mathbb{R}, \, b \in \mathbb{C}\Big\}$$ on which the Pfaffian has signature $(4,2)$ is preserved under conjugation by $\mathrm{SU}_{2,2}(\mathbb{C})$, and this action realizes the isomorphism $\mathrm{SU}_{2,2}(\mathbb{C}) \cong \mathrm{Spin}_{4,2}(\mathbb{R})$. The lattice of $\mathcal{O}_K$-integral matrices (which is isometric to $\mathcal{O}_K \oplus \mathrm{II}_{2,2}$) is preserved by $\mathrm{SU}_{2,2}(\mathcal{O}_K)$ and we obtain an embedding of $\mathrm{SU}_{2,2}(\mathcal{O}_K)$ in the discriminant kernel $\Gamma_{\mathcal{O}_K}$. This isomorphism induces an identification between the homogeneous spaces $\mathbf{H}_2$ and $\mathbb{H}_{\Lambda}$ and allows orthogonal modular forms to be interpreted as Hermitian modular forms of the same weight.

 %If we fix any $\mathbb{C}$-linear isomorphism $\psi = (\psi_1,\psi_2) : \mathcal{O}_K \otimes_{\mathbb{Z}} \mathbb{C} \stackrel{\sim}{\rightarrow} \mathbb{C}^2$ which induces a map $\rho : \mathrm{SU}_{2,2}(\mathcal{O}_K) \rightarrow \mathrm{SO}^+(\Lambda)$ then we obtain a bijection $$\phi : \mathbf{H}_2 \longrightarrow \mathbb{H}_{\mathcal{O}_K}, \; \; \begin{psmallmatrix} \tau & \psi_1(z) \\ \psi_2(z) & w \end{psmallmatrix} \mapsto (\tau,z,w),$$ which is equivariant with respect to both types of M\"obius transformations in the sense that $$\phi \left( M \cdot \tau\right) = \rho(M) \cdot \phi(\tau) \; \text{for all} \; M \in \mathrm{SU}_{2,2}(\mathcal{O}_K) \; \text{and} \; \tau \in \mathbf{H}_2.$$

%Fix a $\mathbb{Z}$-basis $(1,\alpha)$ of $\mathcal{O}_K$ and therefore an isomorphism $\Lambda \otimes \mathbb{C} \cong \mathbb{C}^2$. The representation $\rho : \mathrm{SU}_{2,2}(\mathcal{O}_K) / \{ \pm 1\} \rightarrow \mathrm{SO}^+(\Lambda)$ we obtain is compatible with the bijection $$\phi : \mathbf{H}_2 \longrightarrow \mathbb{H}_{N_{K/\mathbb{Q}}}, \; \; \begin{psmallmatrix} \tau & z_1 + \alpha z_2 \\ z_1 + \overline{\alpha} z_2 & w \end{psmallmatrix} \mapsto (\tau,z_1,z_2,w)$$ in the sense that $\phi(M \cdot Z) = \rho(M) \cdot \phi(Z)$ and $j(\rho(M),\phi(Z)) = \mathrm{det}(cZ + d)$ for $M = \begin{psmallmatrix} a & b \\ c & d \end{psmallmatrix} \in \mathrm{SU}_{2,2}(\mathcal{O}_K)$, such that orthogonal modular forms $f$ for the lattice $N_{K/\mathbb{Q}} \oplus II_{2,2}$ can be interpreted as Hermitian modular forms $f \circ \phi$ of the same weight.

The discriminant kernel $\Gamma_{\mathcal{O}_K}$ contains the involution $\alpha \mapsto \overline{\alpha}$ of $\mathcal{O}_K$ (in other words, $\alpha - \overline{\alpha} \in \mathcal{O}_K$ for all $\alpha$ in the codifferent $\mathcal{O}_K^{\#}$), and this involution does not come from the action of $\mathrm{SU}_{2,2}(\mathcal{O}_K)$. This means that Hermitian modular forms which arise from orthogonal modular forms are either symmetric or skew-symmetric: 

\begin{defn} A Hermitian modular form $F : \mathbf{H}_2 \rightarrow \mathbb{C}$ of weight $k$ is \emph{(graded) symmetric} if $$F(z^T) = (-1)^k F(z) \; \text{for all} \; z \in \mathbf{H}_2,$$ and \emph{(graded) skew-symmetric} if $F(z^T) = -(-1)^k F(z)$.
\end{defn}
Note that many references (e.g. \cite{DK1},\cite{DK2}) use the notion of (skew)-symmetry without respect to the grading, i.e. without the factor $(-1)^k$.

The \emph{maximal} discrete extension $\Gamma_K^*$ of $\Gamma_K$ (as computed in \cite{KRW}) also contains a copy of the class group $\mathrm{Cl}(\mathcal{O}_K)$ which is generally not contained in the discriminant kernel. We only consider the fields $K = \mathbb{Q}(\sqrt{-7}),\mathbb{Q}(\sqrt{-11})$ of class number one so we will not discuss this point further; however, if one were to extend the arguments below to general number fields then most instances of the discrete extension $\Gamma_{\mathcal{O}_K}$ of $\Gamma_K$ below should probably be replaced by $\Gamma_K^*$.

\subsection{Heegner divisors} On orthogonal Shimura varieties there is a natural construction of \emph{Heegner divisors}. Suppose $\Lambda$ is an even lattice of signature $(\ell,2)$. Given any lattice vector $\lambda \in \Lambda$ of positive norm, consider the orthogonal complement $\lambda^{\perp} \cap \mathbb{H}_{\Lambda}$ which has codimension one. The union of these orthogonal complements as $\lambda$ ranges through the (finitely many) primitive lattice vectors of a given norm $D$ is $\Gamma_{\Lambda}$-invariant and defines an analytic cycle $\mathcal{H}_D$ on $\overline{\Gamma_{\Lambda} \backslash \mathbb{H}_{\Lambda}}$. (If we do not take only primitive vectors then we obtain the divisors $\sum_{f^2 | D} \mathcal{H}_{D/f^2}$, which are also often called the Heegner divisors in the literature. For our purposes this definition is less convenient.)

The irreducible components $\mathcal{H}_{D,\pm \beta}$ of $\mathcal{H}_D$ correspond to pairs $(\pm \beta) \in \Lambda'/\Lambda$ of norm $D / \mathrm{disc}(\Lambda)$. In particular when $\mathrm{disc}(\Lambda)$ is prime then every $\mathcal{H}_D$ is irreducible.

Each Heegner divisor is itself an orthogonal Shimura variety for a lattice of signature $(2,\ell - 1)$. (For example, in the Hermitian modular form case the Heegner divisor $\mathcal{H}_D$ may be identified with the paramodular threefold $X_{K(D)}$ of level $D$ modulo Atkin-Lehner involutions.) Moreover the intersection of any two Heegner divisors is itself a Heegner divisor in this interpretation. The intersection numbers can be computed in general by counting certain lattice embeddings up to equivalence. However it seems worthwhile to mention a trick which (in the cases we will need) makes this computation quite easy and which works in some generality.

A special case of Borcherds' higher-dimensional Gross-Kohnen-Zagier theorem \cite{BGKZ} shows that the Heegner divisors on $\Gamma_K \backslash \mathbf{H}_2$ interpreted appropriately are coefficients of a modular form of weight $3$. If $K$ has prime discriminant $d_K < 0$, and we take intersection numbers with a fixed Heegner divisor of squarefree discriminant $m \in \mathbb{N}$ and apply the Bruinier-Bundschuh isomorphism (see \cite{BB}, or Remark 3 below) then this implies that there are weights $\alpha_m(D)$, $D \in \mathbb{N}$ such that $$\Phi_m(\tau) := -1 + \sum_{D=1}^{\infty} \alpha_m(D) \sum_{f^2 | D} (\mathcal{H}_m \cdot \mathcal{H}_{D/f^2}) q^D \in M_3^+(\Gamma_0(-d_K),\chi),$$ where $\chi$ is the quadratic Dirichlet character modulo $d_K$, and where $M_3^+(\Gamma_0(-d_K),\chi)$ is the subspace of weight three modular forms of level $\Gamma_0(-d_K)$ whose Fourier expansions at $\infty$ are supported on exponents which are quadratic residues. Moreover the sums $\sum_{f^2 | D} \alpha_m(D/f^2)$ themselves (for fixed $m$) are coefficients of a modular form of weight $5/2$ and level $\Gamma_0(4m)$ satisfying the  Kohnen plus-condition and which has constant term $-1$ (and for $m=1,2,3$ this determines it uniquely); for example, $$-1 + \sum_{D=1}^{\infty} \sum_{f^2 | D} \alpha_1(D/f^2) q^D = -1 + 10q + 70q^4 + 48q^5 + 120q^8 + 250q^9 + ... = 6 \frac{\theta'(\tau)}{2\pi i} - E_2(4 \tau) \theta(\tau),$$ $$-1 + \sum_{D=1}^{\infty} \sum_{f^2 | D} \alpha_2(D/f^2) q^D = -1 + 4q + 22q^4 + 24q^8 + 100q^9 + ... = 3 \frac{\theta'(\tau)}{2\pi i} - E_2(8\tau) \theta(\tau),$$ $$-1 + \sum_{D=1}^{\infty} \sum_{f^2 | D} \alpha_3(D/f^2) q^D = -1 + 2q + 14q^4 + 34q^9 + 24q^{12} + ... = 2 \frac{\theta'(\tau)}{2\pi i} - E_2(12 \tau) \theta(\tau),$$ where $\theta(\tau) = 1 + 2q + 2q^4 + 2q^9 + ...$ is the usual theta function and where $E_2(\tau) = 1 - 24 \sum_{n=1}^{\infty} \sigma_1(n) q^n$.

Unfortunately the spaces $M_3^+(\Gamma_0(-d_K),\chi)$ are two-dimensional for $d_K \in \{-7,-11\}$. However one can specify the correct modular forms more precisely by observing that the intersections in cohomology are themselves the Fourier coefficients of a vector-valued Jacobi form of index $m/|d_K|$ and weight three (for a particular representation of the Jacobi group) and the intersection numbers are obtained by setting the elliptic variable of that Jacobi form to zero. (More precisely these Jacobi forms occur as Fourier-Jacobi coefficients of the Siegel modular form introduced by Kudla-Millson in \cite{KM}.) For $m \le 3$ the relevant space of Jacobi forms is always one-dimensional (for \emph{every} $d_K$), spanned by the Eisenstein series (for which some computational aspects are discussed in \cite{W}) so the generating series of intersection numbers is exactly what was called the \emph{Poincar\'e square series} of index $m/|d_K|$ in \cite{W}. In this way we can compute the relevant intersection numbers without computing any intersections. We find:

\noindent (1) For $K = \mathbb{Q}(\sqrt{-7})$, $$\Phi_1(\tau) = -1 - 2q + 20q^2 + 18q^4 + 70q^7 + 160q^8 + 94q^9 + ...$$ and $$\Phi_2(\tau) = -1 + 4q + 2q^2 + 48q^4 + 28q^7 + 142q^8 + 148q^9 + ...$$

\noindent (2) For $K = \mathbb{Q}(\sqrt{-11})$, $$\Phi_1(\tau) = -1 - 2q + 20q^3 - 2q^4 + 20q^5 + 18q^9 + 70q^{11} + ...$$ and $$\Phi_3(\tau) = -1 + 2q + 0q^3 + 14q^4 + 16q^5 + 82q^9 + 26q^{11} + ...$$

It follows that for $K = \mathbb{Q}(\sqrt{-7})$, the intersection of $\mathcal{H}_1$ and $\mathcal{H}_2$ as a Heegner divisor of $X_{K(1)}$ is $2H_1$ and as a Heegner divisor of $X_{K(2)}$ is just $H_1$ itself; and for $K = \mathbb{Q}(\sqrt{-11})$ the intersection of $\mathcal{H}_1$ and $\mathcal{H}_3$ in $X_{K(1)}$ is $2 H_1$ and in $X_{K(2)}$ is $H_1$. This means, for example, that if $F$ is a Hermitian modular form for $\mathcal{O}_K$, $K = \mathbb{Q}(\sqrt{-7})$ with a zero on $\mathcal{H}_2$, then the pullbacks of all orders to $\mathcal{H}_1$ are Siegel modular forms of degree two with at least a double zero along the diagonal.

\subsection{Lifts} To construct generators we make use of two lifts from elliptic modular forms: the \emph{Maass lift} (or additive theta lift) and the \emph{Borcherds lift} (or multiplicative theta lift). Both theta lifts most naturally take vector-valued modular forms which transform under a Weil representation as inputs.

Recall that if $(\Lambda,Q)$ is an even-dimensional even lattice with dual $\Lambda'$ then there is a representation $\rho^*$ of $\mathrm{SL}_2(\mathbb{Z})$ on $\mathbb{C}[\Lambda'/\Lambda] = \mathrm{span}(\mathfrak{e}_{\gamma}: \; \gamma \in \Lambda'/\Lambda)$ defined by $$\rho^* \left( \begin{psmallmatrix} 0 & -1 \\ 1 & 0 \end{psmallmatrix} \right) \mathfrak{e}_{\gamma} = \frac{e^{-\pi i \mathrm{sig}(\Lambda)/4}}{\sqrt{|\Lambda'/\Lambda|}} \sum_{\beta \in \Lambda'/\Lambda} e^{2\pi i \langle \beta, \gamma \rangle} \mathfrak{e}_{\beta}, \; \; \rho^* \left( \begin{psmallmatrix} 1 & 1 \\ 0 & 1 \end{psmallmatrix} \right) \mathfrak{e}_{\gamma} = e^{-2\pi i Q(\gamma)} \mathfrak{e}_{\gamma}.$$ We consider holomorphic functions $F : \mathbb{H} \rightarrow \mathbb{C}[\Lambda'/\Lambda]$ which satisfy the functional equations $$F \left( \frac{a \tau + b}{c \tau + d} \right) = (c \tau + d)^k \rho^*\left( \begin{psmallmatrix} a & b \\ c & d \end{psmallmatrix} \right)$$ for all $\begin{psmallmatrix} a & b \\ c & d \end{psmallmatrix} \in \mathrm{SL}_2(\mathbb{Z})$. These are called \emph{nearly-holomorphic modular forms} if they have finite order at $\infty$ (in other words, $F(x+iy)$ has at worst exponential growth as $y \rightarrow \infty$), and are \emph{(holomorphic) modular forms} or \emph{cusp forms} if $F(x+iy)$ is bounded or tends to zero in that limit, respectively. The functional equation under $T = \begin{psmallmatrix} 1 & 1 \\ 0 & 1 \end{psmallmatrix}$ implies a Fourier expansion of the form $$F(\tau) = \sum_{\gamma \in \Lambda'/\Lambda} \sum_{n \in \mathbb{Z}^n - Q(\gamma)} c(n,\gamma) q^n \mathfrak{e}_{\gamma}$$ where $q = e^{2\pi i \tau}$ and $c(n,\gamma) \in \mathbb{C}$. Then $F$ is a nearly-holomorphic modular form if and only if $c(n,\gamma) = 0$ for all sufficiently small $n$; a holomorphic modular form if and only if $c(n,\gamma) = 0$ for all $n < 0$; and a cusp form if and only if $c(n,\gamma) = 0$ for all $n \le 0$.

Now suppose $\Lambda$ is positive-definite and that $k \ge \frac{1}{2} \mathrm{dim}\, \Lambda$, $k \in \mathbb{Z}$. The \emph{Maass lift} takes a vector-valued modular form $F(\tau) = \sum_{\gamma,n} c(n,\gamma) q^n \mathfrak{e}_{\gamma}$ of weight $\kappa = k - \frac{1}{2}\mathrm{dim}\, \Lambda$ for $\rho^*$ to the orthogonal modular form $$\Phi_F(\tau,z,w)  = -\frac{B_k}{2k} c(0,0) \Big( E_k(\tau) + E_k(w) - 1 \Big) + \sum_{a,b=1}^{\infty} \sum_{\substack{\lambda \in \Lambda' \\ \lambda \; \text{positive} \\ Q(\lambda) \le ab}} \sum_{n=1}^{\infty} c(ab - Q(\lambda), \lambda) n^{k-1} e^{2\pi i n (a \tau + b w + \langle \lambda, z \rangle)}$$ for $\Lambda \oplus \mathrm{II}_{2,2}$, where $E_k(\tau), E_k(w)$ denote the Eisenstein series of weight $k$ for $\mathrm{SL}_2(\mathbb{Z})$. (If $k$ is odd then $c(0,0) = 0$ so there is no need to define $E_k$.) The Maass lift is additive and preserves the subspace of cusp forms.

The second lift we use is the Borcherds lift, which takes a \emph{nearly-holomorphic} vector-valued modular form $F(\tau) = \sum_{\gamma,n} c(n,\gamma) q^n \mathfrak{e}_{\gamma}$ of weight $-\frac{1}{2}\mathrm{dim}\, \Lambda$ (where we again take $\Lambda$ to be positive-definite) and yields a multivalued meromorphic orthogonal modular form (in general with character) which is locally represented as a convergent infinite product: $$\Psi_F(\tau,z,w) = e^{2\pi i (A \tau + \langle B,z \rangle + Cw)} \prod_{a,b,\lambda} (1 - e^{2\pi i (a\tau + bw + \langle \lambda,z \rangle)})^{c(ab - Q(\lambda),\lambda)}.$$ There is an analogy to the formal $k=0$ case of the Maass lift; however, the set over which $a,b,\lambda$ is more complicated (depending on a \emph{Weyl chamber} containing $(\tau,z,w)$) and the \emph{Weyl vector} $(A,B,C)$ has no analogue in the additive lift. The most important aspect of the Borcherds lift for us is not the product expansion but the fact that the divisor of $\Psi_F$ may be computed exactly: it is supported on Heegner divisors, and the order of $\Psi_F$ on the rational quadratic divisor $\lambda^{\perp}$ (with $Q(\lambda) < 0$) is $$\mathrm{ord}(\Psi_F; \lambda^{\perp}) = \sum_{r \in \mathbb{Q}_{>0}} c(r^2 Q(\lambda), r\lambda)$$ (where $c(r^2 Q(\lambda),r \lambda) = 0$ if $r \lambda \not \in \Lambda'$). In particular $\Psi_F$ is an orthogonal modular form if and only if these orders are nonnegative integers. In all cases the weight of $F$ is $c(0,0)/2$.

\begin{rem} One can always compactify $\Gamma_{\Lambda} \backslash \mathbb{H}_{\Lambda}$ by including finitely many zero-dimensional and one-dimensional cusps (corresponding to isotropic one-dimensional or two-dimensional sublattices of $\Lambda \oplus \mathrm{II}_{2,2}$ up to equivalence). If $K$ has class number one (or slightly more generally if the norm form on $\mathcal{O}_K$ is alone in its genus) then our discriminant kernel $\Gamma_{\mathcal{O}_K}$ admits only one equivalence class each of zero-dimensional and one-dimensional cusps and both are contained in the closure of every rational quadratic divisor. In particular any Borcherds product which is holomorphic is automatically a cusp form. (This is peculiar to the lattices considered here; it is certainly not true in general.)
\end{rem}

\begin{rem} Let us say a few words about the input functions $F$. A general method to compute vector-valued modular forms for general lattices was given in \cite{W} and \cite{W2} (the two references corresponding to even and odd-weight theta lifts, respectively), and this is what was actually used in the computations below because the implementation was already available. Of course one can obtain all nearly-holomorphic modular forms by dividing true modular forms of an appropriate weight by a power of the discriminant $\Delta(\tau) = q \prod_{n=1}^{\infty} (1 - q^n)^{24}$. However a few other formalisms apply to the particular lattices $\Lambda = (\mathcal{O}_K,N_{K/\mathbb{Q}})$ considered here: \\

\noindent (i) Modular forms for the representation $\rho^*$ attached to a positive-definite lattice $\Lambda$ are equivalent to \emph{Jacobi forms of lattice index} which are scalar-valued functions $\phi(\tau,z)$ in a ``modular variable" $\tau \in \mathbb{H}$ and an ``elliptic variable" $z \in \Lambda \otimes \mathbb{C}$ satisfying certain functional equations and growth conditions. The main advantage of Jacobi forms is that they can be multiplied: for example, in many cases it is possible to construct all Jacobi forms of a given weight and level by taking linear combinations of products of Jacobi theta functions at various arguments (i.e. theta blocks).\\

\noindent (ii) If $\Lambda$ has odd prime discriminant $p$ and $k + (\mathrm{dim}\, \Lambda)/2$ is even then Bruinier and Bundschuh show in \cite{BB} that vector-valued modular forms of weight $k$ for $\rho^*$ can be identified with either a ``plus-" or ``minus-" subspace of $M_k(\Gamma_0(p),\chi_p)$ (where $\chi_p$ is the nontrivial quadratic character mod $p$), i.e. the subspace of modular forms whose Fourier coefficients are supported on quadratic residues modulo $p$, or quadratic nonresidues mod $p$ and $p\mathbb{Z}$, respectively. The isomorphism simply identifies the form $F(\tau) = \sum_{\gamma,n} c(n,\gamma) q^n \mathfrak{e}_{\gamma}$ with $$\sum_{\gamma,n} c(n,\gamma) q^{pn} \in M_k(\Gamma_0(p),\chi_p).$$ This fails when $k + (\mathrm{dim}\, \Lambda)/2$ is odd (in which case $c(n,\gamma) = -c(n,-\gamma)$, so the resulting sum is always zero!). To obtain any results in the the same spirit, it seems necessary to consider instead the ``twisted sums" $$\sum_{\gamma,n} c(n,\gamma) \chi(\gamma) q^{pn},$$ where $\chi$ is an odd Dirichlet character mod $p$ (and where an isomorphism $\Lambda'/\Lambda \cong \mathbb{Z}/p\mathbb{Z}$ has been fixed). The result is a modular form of level $\Gamma_0(p^2)$ with character $\chi \otimes \chi_p$. These maps were studied in \cite{SW}; they are injective and their images can be characterized in terms of the Atkin-Lehner involutions modulo $p^2$.
\end{rem}

\subsection{Pullbacks} Let $\lambda \in \mathcal{O}_K$ have norm $\ell = N_{K/\mathbb{Q}} \lambda$, and consider the embedding of the Siegel upper half-space into $\mathbf{H}_2$: $$\phi : \mathbb{H}_2 \longrightarrow \mathbf{H}_2, \; \; \phi \left( \begin{psmallmatrix} \tau & z \\ z & w \end{psmallmatrix} \right) = \begin{psmallmatrix} \tau & \overline{\lambda} z \\ \lambda z & \ell w \end{psmallmatrix} = U_{\lambda} \cdot \begin{psmallmatrix} \tau & z \\ z & w \end{psmallmatrix}, \; U_{\lambda} := \mathrm{diag}(1,\lambda,1,\lambda/\ell).$$ For any paramodular matrix $$M \in K(\ell) := \{M \in \mathrm{Sp}_4(\mathbb{Q}): \; \sigma_{\ell}^{-1} M \sigma_{\ell} \in \mathbb{Z}^{4 \times 4} \}, \; \sigma_{\ell} := \mathrm{diag}(1,1,1,\ell),$$ we find $U_{\lambda} M U_{\lambda}^{-1} \in \mathrm{SU}_{2,2}(\mathcal{O}_K)$ and $$\phi(M \cdot \tau) = (U_{\lambda} M U_{\lambda}^{-1}) \cdot \phi(\tau), \; \tau \in \mathbb{H}_2,$$ so $\phi$ descends to an embedding of $K(\ell) \backslash \mathbb{H}_2$ into $\Gamma_K \backslash \mathbf{H}_2$ (and more specifically into the Heegner divisor of discriminant $\ell$). In particular if $F : \mathbf{H}_2 \rightarrow \mathbb{C}$ is a Hermitian modular form then $f := F \circ \phi$ is a paramodular form of the same weight, i.e. $$f(M \cdot \tau) = (c \tau + d)^k f(\tau) \; \text{for all} \; M = \begin{psmallmatrix} a & b \\ c & d \end{psmallmatrix} \in K(\ell) \; \text{and} \; \tau \in \mathbb{H}_2.$$

The preprint \cite{W3} gives expressions in the higher Taylor coefficients about a rational quadratic divisor which yield ``higher pullbacks" $P_N F$, $N \in \mathbb{N}_0$. If $F$ is a Hermitian modular form of weight $k$ then its pullback $P_N^{\mathcal{H}_{\ell}} F$ along the embedding above is a paramodular form of level $K(\ell)$ and weight $k+N$ and a cusp form if $N > 0$. The higher pullbacks of theta lifts are themselves theta lifts and are particularly simple to compute. One computational aspect of the higher pullbacks worth mentioning is that a form $F$ vanishes to some order $h$ along the rational quadratic divisor if and only if its pullbacks $P_N F$, $N < h$ are identically zero, and this can be checked rigorously using Sturm bounds (or their generalizations) for the lower-dimensional group under which $P_N F$ transforms.

An important case is the $N^{\mathrm{th}}$ pullback of a modular form $F$ to a Heegner divisor along which it has order exactly $N$. The result in this case is the well-known \emph{quasi-pullback} and we denote it $\mathrm{Q}F$. The quasi-pullback is multiplicative i.e. $\mathrm{Q}(FG) = \mathrm{Q}F \cdot \mathrm{Q}G$ for all Hermitian modular forms $F,G$.

%\noindent \textbf{2.5 Pullbacks.} Let $\Lambda$ be a positive-definite even lattice and let $\lambda \in \Lambda$ be any vector. There is a natural embedding of half-spaces $\phi : \mathbb{H}_{\lambda^{\perp}} \rightarrow \mathbb{H}_{\Lambda}$ which is compatible with the actions of the orthogonal modular groups and yields a \emph{pullback map} between modular forms of the same weight: $$\phi^* : M_k(\Gamma_{\Lambda}) \rightarrow M_k(\Gamma_{\lambda^{\perp}}), \; \; F \mapsto F \circ \phi.$$ In \cite{W3}, expressions in the higher Taylor coefficients about the divisor $\lambda^{\perp}$ are given which yield the ``higher pullbacks'' $P_N F$, $N \in \mathbb{N}$. These are modular forms of weight $k+N$ on $\mathbb{H}_{\lambda^{\perp}}$. The pullbacks are useful when one wants to produce modular forms which are not Borcherds products but have known orders along Heegner divisors. In this section we will specialize \cite{W3} to Hermitian modular forms.

%Any orthogonal modular form $F : \mathbb{H}_{\Lambda} \rightarrow \mathbb{C}$ can be expanded as a Taylor series about the rational quadratic divisor $\lambda^{\perp}$. There is a natural way to form orthogonal modular forms on $\mathbb{H}_{\lambda^{\perp}}$ out of the Taylor coefficients of $F$ which we now describe.

\section{Paramodular forms of levels one, two and three}

The pullbacks of Hermitian modular forms to certain Heegner divisors have interpretations as paramodular forms (as in subsection 2.5 above). Structure results for graded rings of paramodular forms are known for a few values of $N$. We will rely on the previously known generators for the graded rings of paramodular levels 1,2 and 3. The first of these is now classical and was derived by Igusa \cite{Ig}; the second was computed in \cite{IO} by Ibukiyama and Onodera; and the third was computed by Dern \cite{Dern2}. For convenience we express the generators as Gritsenko lifts or Borcherds products. (Igusa and Ibukiyama--Onodera expressed them in terms of thetanulls.)

\begin{prop} (i) There are cusp forms $\psi_{10},\psi_{12},\psi_{35}$ of weights $10,12,35$ such that $M_{\ast}(K(1))$ is generated by the Eisenstein series $E_4,E_6$ and by $\psi_{10},\psi_{12},\psi_{35}$. \\ (ii) There are graded-symmetric cusp forms $\phi_8,\phi_{10},\phi_{11},\phi_{12}$ of weights $8,10,11,12$ and an antisymmetric non-cusp form $f_{12}$ such that $M_{\ast}(K(2))$ is generated by the Eisenstein series $E_4,E_6$ and by $\phi_8,\phi_{10},\phi_{11},\phi_{12},f_{12}$. \\ (iii) There are graded-symmetric cusp forms $\varphi_6,\varphi_8,\varphi_9,\varphi_{10},\varphi_{11},\varphi_{12}$ of weights $6,8,9,10,11,12$ and an antisymmetric non-cusp form $f_{12}$ such that $M_{\ast}(K(3))$ is generated by the Eisenstein series $E_4,E_6$ and by $\varphi_6,\varphi_8,\varphi_9,\varphi_{10},\varphi_{11},\varphi_{12},f_{12}$.
\end{prop}
For later use, we fix the following concrete generators. Let $E_4,E_6$ denote the modular Eisenstein series; $E_{k,m}$ the Jacobi Eisenstein series of weight $k$ and index $m$; and $E_{k,m}'$ its derivative with respect to $z$. The inputs into the Gritsenko and Borcherds lifts are expressed as Jacobi forms following Remark 3 above. \\

(i) $\psi_{10}$ and $\psi_{12}$ are the Gritsenko lifts of the Jacobi cusp forms $$\varphi_{10,1}(\tau,z) = \frac{E_{4,1}E_6 - E_4 E_{6,1}}{144} \; \text{and} \; \varphi_{12,1}(\tau,z) =  \frac{E_4^2 E_{4,1} - E_6 E_{6,1}}{144}$$ respectively, and $\psi_{35}$ is the Borcherds lift of $\frac{11 E_4^2 E_{4,1} + 7 E_6 E_{6,1}}{18\Delta}$. \\

(ii) $\phi_8,\phi_{10},\phi_{11},\phi_{12}$ are the Gritsenko lifts of the Jacobi cusp forms $$\varphi_{8,2} = \frac{E_4 E_{4,2} - E_{4,1}^2}{12}, \; \varphi_{10,2} = \frac{E_{4,2} E_6 - E_{4,1}E_{6,1}}{12}, \; \varphi_{11,2} = \frac{E_{4,1} E_{6,1}' - E_{4,1} E_{6,1}'}{288\pi i}, \; \varphi_{12,2} = \frac{E_4^2 E_{4,2} - E_6 E_{6,2}}{24},$$ respectively, and $f_{12}$ is the Borcherds lift of $\frac{3 E_4^2 E_{4,2} + 4 E_4 E_{4,1}^2 + 5 E_6 E_{6,2}}{12 \Delta}$.  \\

(iii) $\varphi_6,\varphi_8,\varphi_9,\varphi_{10},\varphi_{11},\varphi_{12}$ are the Gritsenko lifts of the Jacobi cusp forms $$\varphi_{6,3} = \frac{\varphi_{10,1}\varphi_{8,2}}{\Delta}, \; \varphi_{8,3} = \frac{E_4 E_{4,3} - E_{4,1} E_{4,2}}{2}, \; \varphi_{9,3} = \frac{\varphi_{10,1} \varphi_{11,2}}{\Delta},$$ $$\varphi_{10,3} = \frac{\varphi_{10,2} \varphi_{12,1}}{\Delta}, \; \varphi_{11,3} = \frac{\varphi_{11,2}\varphi_{12,1}}{\Delta}, \; \varphi_{12,3} = \frac{E_4 E_{4,1}E_{4,2} + E_4^2 E_{4,3}}{2} - E_{6,1} E_{6,2},$$ respectively, and $f_{12}$ is the Borcherds lift of $\frac{2 E_4 E_{4,1} E_{4,2} + 5 E_{4,1}^3 + 5E_{6,1} E_{6,2}}{12 \Delta}.$ (Note that these are not quite the generators used by Dern; the choices used here simplify the ideal of relations somewhat.)

\begin{rem} For later use we will need to understand the ideals of symmetric (under the Fricke involution $\tau \mapsto -\frac{1}{N} \tau^{-1}$) paramodular forms of level $N \in \{1,2,3\}$ which vanish along the diagonal. The pullback of a paramodular form to the diagonal is a modular form for the group $\mathrm{SL}_2(\mathbb{Z}) \times \mathrm{SL}_2(\mathbb{Z})$ or in other words a linear combination of expressions of the form $(f_1 \otimes f_2)(\tau_1,\tau_2) = f_1(\tau_1)f_2(\tau_2)$, where $f_1,f_2$ are elliptic modular forms of level one of the same weight; and if the paramodular form is symmetric then the pullback is symmetric under swapping $(\tau_1,\tau_2) \mapsto (\tau_2,\tau_1)$. The graded ring of symmetric modular forms under $\mathrm{SL}_2(\mathbb{Z}) \times \mathrm{SL}_2(\mathbb{Z})$ is the weighted polynomial ring $$M_{\ast}(\mathrm{SL}_2(\mathbb{Z}) \times \mathrm{SL}_2(\mathbb{Z})) = \mathbb{C}[E_4 \otimes E_4, E_6 \otimes E_6, \Delta \otimes \Delta]$$ where $E_4,E_6,\Delta$ are defined as usual. Therefore: \\

\noindent (i) In level $N=1$, the pullbacks of $E_4,E_6,\psi_{12}$ to the diagonal are the algebraically independent modular forms $E_4 \otimes E_4$, $E_6 \otimes E_6$, $\Delta \otimes \Delta$, so every even-weight form which vanishes on the diagonal is a multiple of $\psi_{10}$ (which has a double zero). The odd-weight form $\psi_{35}$ has a simple zero on the diagonal. \\

\noindent (ii) In level $N=2$, the pullbacks of $E_4,E_6,\phi_{12}$ to the diagonal are algebraically independent, so the ideal of even-weight symmetric forms which vanish on the diagonal is generated by $\phi_8$ (which has a fourth-order zero there) and $\phi_{10}$ (which has a double zero). Moreover $\phi_{10}^2$ is itself a multiple of $\phi_8$, so the ideal of even-weight modular forms which vanish to order at least three along the diagonal is principal, generated by $\phi_8$. The odd-weight form $\phi_{11}$ has a simple zero along the diagonal. \\

\noindent (iii) In level $N=3$, the pullbacks of $E_4,E_6,\varphi_{12}$ to the diagonal are algebraically independent, so the ideal of even-weight symmetric forms which vanish on the diagonal is generated by $\varphi_6,\varphi_8,\varphi_{10}$ (which have zeros of order $6,4,2$ respectively). These forms satisfy $\varphi_8^2 = \varphi_6 \varphi_{10}$ and $\varphi_{10}^2 = \varphi_8 \varphi_{12}$, so the ideals of (even-weight, symmetric) forms which vanish to order at least $3$ or at least $5$ are $\langle \varphi_6,\varphi_8 \rangle$ and $\langle \varphi_6 \rangle$, respectively. The odd-weight forms $\varphi_9$ and $\varphi_{11}$ have order $3$ and $1$ along the diagonal, respectively, and satisfy the relations $$\varphi_6 \varphi_{11} = \varphi_8 \varphi_9, \; \varphi_8 \varphi_{11} = \varphi_9 \varphi_{10},$$ and $\varphi_{11}^3$ and $\varphi_{10} \varphi_{11}$ (and therefore all odd-weight symmetric forms with at least a triple zero on the diagonal) are multiples of $\varphi_9$.
\end{rem}

\section{Hermitian modular forms for $\mathbb{Q}(\sqrt{-7})$}

In this section we compute the graded ring of Hermitian modular forms for the maximal order in $K = \mathbb{Q}(\sqrt{-7})$ by studying the pullbacks to Heegner divisors of discriminant $1$ and $2$ and applying the structure theorems of Igusa and Ibukiyama-Onodera. We first consider graded-symmetric forms and reduce against a distinguished Borcherds product $b_7$ (which is also a Maass lift) whose divisor is $$\mathrm{div} \, b_7 = 3\mathcal{H}_1 + \mathcal{H}_2.$$ We will express all graded-symmetric forms in terms of Maass lifts $\mathcal{E}_4,\mathcal{E}_6,b_7,m_8,m_9,m_{10}^{(1)},m_{10}^{(2)},m_{11},m_{12}$ in weights $4,6,7,8,9,10,10,11,12$ which are described in more detail on the next page. The Maass lifts of weight $4,6,7,8,9$ are essentially unique, and the Maass lifts of weight $10$ are chosen such that $m_{10}^{(1)}$ vanishes on $\mathcal{H}_1$ and $m_{10}^{(2)}$ vanishes on $\mathcal{H}_2$. By contrast $m_{11}$ could have been chosen almost arbitrarily (so long as it is not a multiple of $\mathcal{E}_4 b_7$, which is also a Maass lift), and similarly for $m_{12}$.

\afterpage{%
\clearpage% Flush earlier floats (otherwise order might not be correct)
\begin{landscape}% Landscape page
In Table 1 we describe the even-weight Maass lifts used as generators. For each Maass lift of weight $k$ we give its input form (in the convention of Bruinier-Bundschuh; this is a modular form of weight $k-1$ and level $\Gamma_0(7)$ for the quadratic character) and its first pullbacks to the Heegner divisors of discriminant $1$ and $2$. (The pullbacks of odd order to $\mathcal{H}_1$ are always zero and therefore omitted.)
\begin{table}[htbp]
\centering
\caption{Maass lifts in even weight}
\begin{tabular}{l*{8}{c}r}
\hline
Name & Weight & Input form & $P_0^{\mathcal{H}_1}$ & $P_2^{\mathcal{H}_1}$ & $P_4^{\mathcal{H}_1}$ & $P_0^{\mathcal{H}_2}$ & $P_1^{\mathcal{H}_2}$\\
\hline
\hline
$\mathcal{E}_4$  & 4 & $1 + 14q^3 + 42q^5 + 70q^6 + 42q^7 + 210q^{10} \pm ...$ & $E_4$ & $0$ & $0$ & $E_4$ & $0$ \\
\hline
$\mathcal{E}_6$   & 6 & $1 - 10q^3 - 78q^5 - 170q^6 - 150q^7 - 1326q^{10} \pm ...$ & $E_6$ & $0$ & $1814400\psi_{10}$ & $E_6$ & $0$ \\
\hline
$m_8$ & $8$ & $q^3 - q^5 - 8q^6 + 7q^7 + 8q^{10} \pm  ...$ & $0$ & $120 \psi_{10}$ & $4352\psi_{12}$ & $2 \phi_8$ & $0$ \\
\hline
$m_{10}^{(1)}$ & $10$ & $q^3 - q^5 + 16q^6 - 17q^7 - 136q^{10} \pm ... $ & $0$ & $152 \psi_{12}$ & $8736 E_4 \psi_{10}$ & $2 \phi_{10}$ & $24\psi_{11}$ \\
\hline
$m_{10}^{(2)}$ & $10$ & $q^5 - q^6 - q^7 + q^{10} - 16q^{12} \pm ...$ & $2 \psi_{10}$ & $-2\psi_{12}$ & $-420E_4 \psi_{10}$ & 0 & $-4\psi_{11}$ \\
\hline
$m_{12}$ & $12$ & $q^5 + 3q^6 + 7q^7 - 19q^{10} - 72q^{12} \pm ...$ & $2 \psi_{12}$ & $2 E_4 \psi_{10}$ & $134E_4 \psi_{12} - 710E_6 \psi_{10}$  & $\frac{1}{3}\phi_{12} - \frac{1}{3}E_4 \phi_8$ & $0$ \\
\hline
\end{tabular}    
\end{table}

The input functions into the Maass lift in odd weight are given as twisted sums as in \cite{SW}. Here, $\chi$ may be any odd Dirichlet character mod $7$; the input form is then a modular form of level $\Gamma_0(49)$ and character $\chi \otimes \chi_7$ where $\chi_7$ is the quadratic character. The Borcherds product $b_7$ happens to lie in the Maass Spezialschar and is listed in this table.

\begin{table}[htbp]
\centering
\caption{Maass lifts in odd weight}
\begin{tabular}{l*{8}{c}r}
\hline
Name & Weight & Input form & $P_1^{\mathcal{H}_1}$ & $P_3^{\mathcal{H}_1}$ & $P_5^{\mathcal{H}_1}$ & $P_0^{\mathcal{H}_2}$ & $P_1^{\mathcal{H}_2}$\\
\hline
\hline
$b_7$  & $7$ & $\chi(5)q^3 + 3\chi(3)q^5 + 2 \chi(1)q^6 - 6 \chi(5)q^{10} \pm ...$ & $0$ & $-360\psi_{10}$ & $4080\psi_{12}$ & $0$ & $-4\phi_8$ \\
\hline
$m_9$ & $9$ & $\chi(5)q^3 - 9\chi(3)q^5 - 10\chi(1)q^6 - 90\chi(5)q^{10} \pm ...$ & $-24 \psi_{10}$ & $72 \psi_{12}$ & $-21168E_4 \psi_{10}$ & $0$ & $-4\phi_{10}$ \\
\hline
$m_{11}$ & $11$ & $\chi(3)q^5 - 5\chi(1)q^6 + 11\chi(5)q^{10} - 30\chi(3)q^{12} \pm ...$ & $-2\psi_{12}$ & $40 E_4 \psi_{10}$ & $\frac{6290}{3}E_4 \psi_{12} - \frac{9350}{3} E_6 \psi_{10} $ & $6 \phi_{11}$ & $2\phi_{12}$ \\
\hline
\end{tabular}    
\end{table}

The Borcherds products below can be shown to exist by a Serre duality argument as in \cite{BGKZ}.

\begin{table}[htbp]
\centering
\caption{Borcherds products}
\begin{tabular}{l*{4}{c}r}
\hline
Name & Weight & Divisor & Graded-symmetric? \\
\hline
\hline
$b_7$ & $7$ & $3\mathcal{H}_1 + \mathcal{H}_2$ & yes \\
\hline
$b_{28}$ & $28$ & $7\mathcal{H}_1 + \mathcal{H}_7$ & no \\
\hline
\end{tabular}    
\end{table}

\end{landscape}

\clearpage% Flush page
}

\begin{lem} Let $F$ be a symmetric Hermitian modular form. There is a polynomial $P$ such that $$F - P(\mathcal{E}_4,\mathcal{E}_6,m_8,m_{10}^{(1)},m_{11},m_{12})$$ vanishes along the Heegner divisor $\mathcal{H}_2$.
\end{lem}
\begin{proof} This amounts to verifying that the pullbacks of $\mathcal{E}_4,\mathcal{E}_6,m_8,m_{10}^{(1)},m_{11},m_{12}$ generate the ring of symmetric paramodular forms of level $2$, and is clear in view of Ibukiyama-Onodera's structure result and Tables 1 and 2 below.
\end{proof}

\begin{thrm} The graded ring of symmetric Hermitian modular forms for $\mathcal{O}_K$ is generated by Maass lifts $$\mathcal{E}_4,\mathcal{E}_6,b_7,m_8,m_9,m_{10}^{(1)},m_{10}^{(2)}, m_{11}, m_{12}$$ in weight $4,6,7,8,9,10,10,11,12$. The ideal of relations is generated by \begin{align*} m_8 m_9 &= b_7 (m_{10}^{(1)} + 12 m_{10}^{(2)}); \\ m_9^2 + 12 b_7 m_{11} &= \mathcal{E}_4 b_7^2 + 36 m_8 m_{10}^{(2)}; \\ m_9 m_{10}^{(1)} &= b_7 (\mathcal{E}_4 m_8 + 12 m_{12}); \\ \mathcal{E}_6 b_7^2 + 18 m_{10}^{(1)} m_{10}^{(2)} &= \mathcal{E}_4 b_7 m_9 + 6 m_9 m_{11}; \\ m_{10}^{(1)} (m_{10}^{(1)} + 12 m_{10}^{(2)}) &= m_8 (\mathcal{E}_4 m_8 + 12 m_{12}); \\ \mathcal{E}_4 b_7 m_{10}^{(1)} + 6 \mathcal{E}_4 b_7 m_{10}^{(2)} + 72 m_{10}^{(2)} m_{11} &= \mathcal{E}_6 b_7 m_8 + 6 m_9 m_{12}; \\ 3 \mathcal{E}_4 m_8 m_{10}^{(1)} + 6 \mathcal{E}_4 b_7 m_{11} + \mathcal{E}_6 b_7 m_9 + 72 m_{11}^2 &= \mathcal{E}_4^2 b_7^2 + 3 \mathcal{E}_6 m_8^2 + 18 m_{10}^{(1)} m_{12}. \end{align*}
\end{thrm}
\begin{proof} We use induction on the weight. As usual any modular form of negative or zero weight is constant. \\

Using the previous lemma we may assume that $F$ has a zero along $\mathcal{H}_2$. Since $\mathcal{H}_2$ has a double intersection with $\mathcal{H}_1$ along its diagonal $H_1$ it follows that the pullbacks of $F$ to $\mathcal{H}_1$ of all orders have (at least) a double zero along the diagonal; in particular, they are multiples of the Igusa discriminant $\psi_{10}$. 

Since the pullbacks of $\mathcal{E}_4,\mathcal{E}_6,m_{10}^{(2)},m_{12}$ to $\mathcal{H}_1$ generate the graded ring of even-weight Siegel modular forms, and $m_{10}^{(2)}$ vanishes along $\mathcal{H}_2$ but pulls back to the Igusa form $\psi_{10}$ on $\mathcal{H}_1$, it follows that we can subtract some expression of the form $$m_{10}^{(2)} P(\mathcal{E}_4,\mathcal{E}_6,m_{10}^{(2)},m_{12})$$ away from $F$ to obtain a form whose pullbacks to both $\mathcal{H}_1$ and $\mathcal{H}_2$ are zero. Similarly, we can subtract some expression of the form $$m_9 P(\mathcal{E}_4,\mathcal{E}_6,m_{10}^{(2)},m_{12})$$ away from $F$ to ensure that the zero along $\mathcal{H}_1$ has multiplicity at least two. \\

Now assume that $F$ has exactly a double zero along $\mathcal{H}_1$ (in particular, it must have even weight) and a zero along $\mathcal{H}_2$. Suppose first that $F$ has exactly a simple zero along $\mathcal{H}_2$. Then its first pullback $P_1^{\mathcal{H}_2} F$ has odd weight and at least a double zero along the diagonal in $X_{K(2)}$ and is therefore contained in the ideal generated by $\phi_8 \phi_{11}$ and $\phi_{10}\phi_{11}$. The products $m_8 m_{10}^{(2)}$ and $m_{10}^{(1)} m_{10}^{(2)}$ have (up to a constant multiple) exactly these first pullbacks, so subtracting away some expression of the form $$m_8 m_{10}^{(2)} P_1(\mathcal{E}_4,\mathcal{E}_6,m_8,m_{10}^{(1)},m_{11},m_{12}) + m_{10}^{(1)} m_{10}^{(2)} P_2(\mathcal{E}_4,\mathcal{E}_6,m_8,m_{10}^{(1)},m_{11},m_{12})$$ with polynomials $P_1,P_2$ leaves us with a modular form with at least double zeros along both $\mathcal{H}_1$ and $\mathcal{H}_2$. The double zero along $\mathcal{H}_2$ forces the second pullback to $\mathcal{H}_1$ to have at least a \emph{fourth}-order zero along the diagonal and therefore to be a multiple of $\psi_{10}^2$. Since $m_9^2$ has exactly this second pullback to $\mathcal{H}_1$ (up to a constant multiple) and a double zero along $\mathcal{H}_2$, we may subtract away some expression of the form $$m_9^2 P(\mathcal{E}_4,\mathcal{E}_6,m_{10}^{(2)},m_{12})$$ from $F$ to obtain a form with a third-order zero along $\mathcal{H}_1$ and which continues to have a double zero on $\mathcal{H}_2$.

Finally, any modular form $F$ with a triple zero along $\mathcal{H}_1$ and a zero along $\mathcal{H}_2$ is divisible by $b_7$ (by Koecher's principle), with the quotient $\frac{F}{b_7}$ having strictly lower weight. By induction, $F/b_7$ and therefore $F$ is a polynomial expression in the generators in the claim.

The relations were computed by working directly with Fourier expansions. Here the main difficulties are determining how many Fourier coefficients must be computed to show that a modular form is identically zero, and determining how many relations are needed to generate the full ideal. To verify the correctness of these computations in both cases it is enough to know the dimensions of spaces of Hermitian modular forms, and these are derived in section 6 below.
\end{proof}

\begin{prop} There are holomorphic skew-symmetric forms $h_{30},h_{31},h_{32},h_{33},h_{34},h_{35}$, which are obtained from $b_{28}$ and the Maass lifts constructed above by inverting $b_7$, such that every Hermitian modular form for $\mathcal{O}_K$ is a polynomial in $$\mathcal{E}_4,\mathcal{E}_6,b_7,m_8,m_9,m_{10}^{(1)},m_{10}^{(2)},m_{11},m_{12},b_{28},h_{30},h_{31},h_{32},h_{33},h_{34},h_{35}.$$
\end{prop}
\begin{proof} As a skew-symmetric form, $F$ has a forced zero on the Heegner divisor $\mathcal{H}_7$. If $F$ has even weight, the point will be to subtract away skew-symmetric forms from $F$ to produce something with at least a seventh-order zero on the surface $\mathcal{H}_1$, which will therefore be divisible by $b_{28}$. By contrast if $F$ has odd weight then it seems to be more effective to reduce first against the product $b_7$.

(i) Suppose $F$ has even weight, so its order along $\mathcal{H}_1$ is odd and its quasi-pullback to $\mathcal{H}_1$ takes the form $$\mathrm{Q}F = \psi_{35} P(\psi_4,\psi_6,\psi_{10},\psi_{12})$$ for some polynomial $P$. The quotients $h_{30} := b_{28} \frac{m_9}{b_7}, h_{32}  := b_{28} \frac{m_9^2}{b_7^2}, h_{34} := b_{28} \frac{m_9^3}{b_7^3}$ are holomorphic and skew-symmetric, with zeros along $\mathcal{H}_1$ of order $5,3,1$ respectively, and in all cases their quasi-pullback to $\mathcal{H}_1$ is a constant multiple of $\psi_{35}$. By subtracting from $F$ expressions of the form $$\{h_{30},h_{32},h_{34}\} \cdot P(\mathcal{E}_4,\mathcal{E}_6,m_{10}^{(2)},m_{12}),$$ we are able to force the first, third and fifth order pullbacks of $F$ to $\mathcal{H}_1$ to vanish. But then $F$ is divisible by $b_{28}$ with symmetric quotient, so we apply the previous proposition. \\

(ii) Suppose $F$ has odd weight (and therefore even order along $\mathcal{H}_1$). Then we will find expressions to subtract away from $F$ to force divisibility by $b_7$. (The reduction against $b_{28}$ as in the even-weight case seems impossible, as there are no skew-symmetric modular forms of weight $29$ and therefore no way to handle sixth-order zeros on $\mathcal{H}_1$.) We will first force $F$ to have at least a fourth-order zero along $\mathcal{H}_1$. The quotients $$h_{33} := \frac{b_{28} m_{10}^{(2)} m_9}{b_7^2}, \; h_{35} := \frac{b_{28} m_{10}^{(2)} m_9^2}{b_7^3}$$ are holomorphic and skew-symmetric, with zeros along $\mathcal{H}_1$ of orders $2$ and $0$, respectively, and their quasi-pullbacks to $\mathcal{H}_1$ are again constant multiples of $\psi_{35}$. By subtracting from $F$ expressions of the form $$\{h_{33},h_{35}\} \cdot P(\mathcal{E}_4,\mathcal{E}_6,m_{10}^{(2)},m_{12}),$$ we can ensure that the $0^{\mathrm{th}}$ and $2^{\mathrm{nd}}$ pullbacks of $F$ to $\mathcal{H}_1$ vanish, so $\mathrm{ord}_{\mathcal{H}_1}(F) \ge 4$. \\

Now the pullback of $F$ to $\mathcal{H}_2$ is skew-symmetric, has odd weight, and vanishes on the diagonal to order at least four, so it is therefore a multiple of the weight $31$ form $\phi_8 \phi_{11} f_{12}$: i.e. $$F\Big|_{\mathcal{H}_2} = \phi_8 \phi_{11} f_{12} P(E_4,E_6,\phi_8,\phi_{10},\phi_{12})$$ for some polynomial $P$. But the form $$h_{31} := \frac{b_{28} m_{10}^{(2)}}{b_7}$$ is holomorphic and skew-symmetric, with a fourth-order zero on $\mathcal{H}_1$, and it restricts to (a multiple of) $\phi_8 \phi_{11} f_{12}$ on $\mathcal{H}_2$. Therefore, some expression of the form $$F - h_{31}  P(\mathcal{E}_4,\mathcal{E}_6,m_8,m_{10}^{(1)},m_{12})$$ has a zero on $\mathcal{H}_2$ and continues to have at least a fourth-order zero on $\mathcal{H}_1$. The result will be divisible by $b_7$ with the quotient having even weight and therefore being covered by case (i).
\end{proof}

\section{Hermitian modular forms for $\mathbb{Q}(\sqrt{-11})$}

In this section we reduce the computation of the graded ring of Hermitian modular forms of degree two for the maximal order in $\mathbb{Q}(\sqrt{-11})$ to the results of Igusa and Dern on paramodular forms. The argument is very nearly the same as the previous section. We first deal with symmetric Hermitian modular forms (of all weights) by reduction against the distinguished Borcherds product $b_5$ with divisor $$\mathrm{div}\, b_5 = 5 \mathcal{H}_1 + \mathcal{H}_3.$$ The Maass lifts we take as generators are described in more detail in the tables on the next page.

\afterpage{
\clearpage% Flush earlier floats (otherwise order might not be correct)

\begin{landscape}% Landscape page

\begin{footnotesize}

As in the previous section, the input forms into the Maass lift in Tables 4 and 5 are expressed as component sums using the convention of \cite{BB} and \cite{SW}. The Borcherds products $b_5,b_8,b_9$ satisfy the Maass condition so they are listed both as Maass lifts and Borcherds products.

\begin{table}[htbp]
\centering
\caption{Maass lifts in even weight}
\begin{tabular}{l*{8}{c}r}
\hline
Name & Weight & Input form & $P_0^{\mathcal{H}_1}$ & $P_2^{\mathcal{H}_1}$ & $P_4^{\mathcal{H}_1}$ & $P_0^{\mathcal{H}_3}$ & $P_1^{\mathcal{H}_3}$\\
\hline
\hline
$\mathcal{E}_4$  & 4 & $1 + 2q^2 + 20q^6 + 32q^7 + 34q^8 + 52q^{10} + ...$ & $E_4$ & $0$ & $0$ & $E_4$ & $0$ \\
\hline
$\mathcal{E}_6$   & 6 & $1 - \frac{22}{85}q^2 - \frac{1804}{85}q^6 - \frac{704}{17}q^7 - \frac{5654}{85}q^8 - \frac{13772}{85}q^{10} - ... $ & $E_6$ & $0$ & $-\frac{26345088}{17}\psi_{10}$ & $E_6 - \frac{266112}{5185}\varphi_6$ & $0$ \\
\hline
$m_6$ & $6$ & $q^2 - 3q^6 - 10q^7 + 2q^8 + 31q^{10} \pm ...$ & $0$ & $0$ & $37440\psi_{10}$ & $2\varphi_6$ & $0$ \\
\hline
$m_8$ & $8$ & $q^2 - 3q^6 + 14q^7 + 2q^8 - 65q^{10} \pm ...$ & $0$ & $0$ & $62016\psi_{12}$ & $2\varphi_8$ & $48\varphi_9$ \\
\hline
$b_8$ & $8$ & $q^6 - q^7 - q^8 + q^{11} + q^{13} \pm ...$ & $0$ & $120 \psi_{10}$ & $-544\psi_{12}$ & $0$ & $-6\varphi_9$ \\
\hline
$m_{10}^{(1)}$ & $10$ & $q^6 + 3q^7 - q^8 + 8q^{10} - 11q^{11} - 27q^{13} \pm ...$ & $0$ & $152 \psi_{12}$ & $2688E_4 \psi_{10}$ & $-\frac{1}{6}E_4 \varphi_6 + \frac{1}{6} \varphi_{10}$ & $2\varphi_{11}$ \\
\hline
$m_{10}^{(2)}$ & $10$ & $q^6 - q^7 + 11q^8 - 12q^{10} - 11q^{11} + q^{13} \pm ...$ & $24\psi_{10}$ & $-16\psi_{12}$ & $240 E_4 \psi_{10}$ & $0$ & $-6\varphi_{11}$ \\
\hline
$m_{12}$ & $12$ & $q^2 + 136q^6 - 77q^7 + 7q^8 + 463q^{10} \pm ...$ & $288 \psi_{12}$ & $24136 E_4 \psi_{10}$ & $1023040E_6 \psi_{10} + 34784E_4 \psi_{12}$ & $2\varphi_{12}$ & $-690E_4 \varphi_9$ \\
\hline
\end{tabular}    
\end{table}

\begin{table}[htbp]
\centering
\caption{Maass lifts in odd weight}
\begin{tabular}{l*{8}{c}r}
\hline
Name & Weight & Input form & $P_1^{\mathcal{H}_1}$ & $P_3^{\mathcal{H}_1}$ & $P_5^{\mathcal{H}_1}$ & $P_0^{\mathcal{H}_3}$ & $P_1^{\mathcal{H}_3}$\\
\hline
\hline
$b_5$ & $5$ & $\chi(8)q^2 - 5\chi(7)q^6 + 4\chi(2)q^7 + 10\chi(6)q^8 - 5\chi(1)q^{10} \pm ...$ & $0$ & $0$ & $187200\psi_{10}$ & $0$ & $6\varphi_6$ \\
\hline
$m_7$ & $7$ & $\chi(8)q^2 + 7\chi(7)q^6 - 8\chi(2)q^7 - 26\chi(6)q^8 + 19\chi(1)q^{10} \pm ...$ & $0$ & $4320\psi_{10}$ & $-32640\psi_{12}$ & $0$ & $6\varphi_8$ \\
\hline
$b_9$ & $9$ & $\chi(2)q^7 - \chi(6)q^8 - \chi(1)q^{10} + \chi(8) q^{13} + \chi(7)q^{17} \pm ...$ & $\psi_{10}$ & $-40\psi_{12}$ & $2472E_4 \psi_{10}$ & $\frac{1}{2}\varphi_9$ & $\frac{1}{24}E_4 \varphi_6 - \frac{1}{24}\varphi_{10}$ \\ 
\hline
$m_9$ & $9$ & $\chi(8)q^2 + 19\chi(7)q^6 - 20\chi(2)q^7 + 82\chi(6)q^8 - 101\chi(1)q^{10} \pm ...$ & $288\psi_{10}$ & $-576\psi_{12}$ & $570816E_4 \psi_{10}$ & $0$ & $\varphi_{10}$ \\
\hline
$m_{11}$ & $11$ & $\chi(2)q^7 + 2\chi(6)q^8 + 8 \chi(1)q^{10} - 17\chi(8)q^{13} - 29\chi(7)q^{17} \pm ...$ & $4 \psi_{12}$ & $-152E_4 \psi_{10}$ & $\frac{8980}{3}E_4 \psi_{12} + \frac{17300}{3}E_6 \psi_{10}$ & $\varphi_{11}$ & $\frac{1}{8}E_4 \varphi_8 - \frac{5}{36} E_6 \varphi_6 - \frac{4560}{61} \varphi_6^2 + \frac{1}{72} \varphi_{12}$ \\
\hline
\end{tabular}    
\end{table}

\begin{table}[htbp]
\centering
\caption{Borcherds products}
\begin{tabular}{l*{4}{c}r}
\hline
Name & Weight & Divisor & Graded-symmetric? \\
\hline
\hline
$b_5$ & $5$ & $5\mathcal{H}_1 + \mathcal{H}_3$ & yes \\
\hline
$b_8$ & $8$ & $2\mathcal{H}_1 + \mathcal{H}_3 + \mathcal{H}_4$ & yes \\
\hline
$b_9$ & $9$ & $\mathcal{H}_1 + \mathcal{H}_5$ & yes \\
\hline
$b_{24}$ & $24$ & $11\mathcal{H}_1 + \mathcal{H}_{11}$ & no \\
\hline
\end{tabular}    
\end{table}

\end{footnotesize}

\end{landscape}
\clearpage% Flush page
}

\begin{lem} Let $F$ be a symmetric Hermitian modular form. There is a polynomial $P$ such that $$F - P(\mathcal{E}_4,\mathcal{E}_6,m_6,m_8,b_9,m_{10}^{(1)},m_{11},m_{12})$$ vanishes along the Heegner divisor $\mathcal{H}_3$.
\end{lem}
\begin{proof} We only need to check that the pullbacks of $\mathcal{E}_4,\mathcal{E}_6,m_6,m_8,b_9,m_{10}^{(1)},m_{11},m_{12}$ to $\mathcal{H}_3$ generate the graded ring of paramodular forms of level $3$. This is clear from Tables 4 and 5 below after comparing the pullbacks with the generators found by Dern as described in Section 3.
\end{proof}

\newpage

\begin{thrm} The graded ring of symmetric Hermitian modular forms for $\mathcal{O}_K$ is generated by Maass lifts $$\mathcal{E}_4,b_5,\mathcal{E}_6,m_6,m_7,b_8,m_8,b_9,m_9,m_{10}^{(1)},m_{10}^{(2)},m_{11},m_{12}$$ in weights $4,5,6,6,7,8,8,9,9,10,10,11,12$.
\end{thrm}
The ideal of relations is considerably more complicated than the analogous ideal for $K = \mathbb{Q}(\sqrt{-7})$ so it is left to an auxiliary file for convenience.
\begin{proof} We use induction on the weight. Any modular form of nonpositive weight is constant. \\

Let $F$ be any symmetric Hermitian modular form. Using the previous lemma we assume that $F$ has a zero along $\mathcal{H}_3$. Then the pullbacks of $F$ to $\mathcal{H}_1$ of all orders have at least a double zero along the diagonal and are therefore multiples of $\psi_{10}$.

The pullbacks of $\mathcal{E}_4,\mathcal{E}_6,m_{10}^{(2)},m_{12}$ to $\mathcal{H}_1$ generate the ring of even-weight Siegel modular forms of degree two. Moreover, the forms $m_{10}^{(2)},m_9,b_8,m_7$ vanish along $\mathcal{H}_3$ and their quasi-pullbacks to $\mathcal{H}_1$ are scalar multiples of $\psi_{10}$. By successively subtracting away from $F$ expressions of the form $$\{m_{10}^{(2)},m_9,b_8,m_7\} \cdot P(\mathcal{E}_4,\mathcal{E}_6,m_{10}^{(2)},m_{12})$$ with appropriately chosen polynomials $P$, we may set the zeroth, first, second and third order pullbacks to $\mathcal{H}_1$ equal to zero while maintaining a zero on the divisor $\mathcal{H}_3$.

Therefore, we may assume that $F$ has at least a fourth-order zero on $\mathcal{H}_1$ and a zero on $\mathcal{H}_3$. Suppose $F$ has exactly a fourth-order zero on $\mathcal{H}_1$. (In particular, $F$ has even weight.) Then the quasi-pullback $\mathrm{Q}F$ of $F$ to $\mathcal{H}_3$ is an odd-weight paramodular form of level $3$ with at least a fourth-order zero on the diagonal, so $\mathrm{Q}F$ is a multiple of $\varphi_9$ and $\mathrm{Q}F / \varphi_9$ is contained in the ideal $\langle \varphi_6,\varphi_8,\varphi_{10} \rangle$ of symmetric paramodular forms of even weight with a zero on the diagonal. Then we can write $$\mathrm{Q}F = \varphi_6 \varphi_9 P_1 + \varphi_8 \varphi_9 P_2 + \Big( -\frac{1}{6} E_4 \varphi_6 + \frac{1}{6} \varphi_{10}\Big) \varphi_9 P_3$$ for some even-weight symmetric paramodular forms $P_1,P_2,P_3$. Since $m_6 b_8$, $m_8 b_8$ and $m_{10}^{(1)} b_8$ have fourth-order zeros on $\mathcal{H}_1$ and are zero on $\mathcal{H}_3$ with respective quasi-pullbacks $\varphi_6 \varphi_9$, $\varphi_8 \varphi_9$ and $(-1/6 E_4 \varphi_6 + \varphi_{10}/6) \varphi_9$, we can take any symmetric forms $\tilde P_1, \tilde P_2, \tilde P_3$ whose pullbacks to $\mathcal{H}_3$ are $P_1,P_2,P_3$ (some polynomials in $\mathcal{E}_4,\mathcal{E}_6,m_6,m_8,b_9,m_{10}^{(1)},m_{11},m_{12}$ will do) and subtract away $$b_8 \cdot \Big( m_6 \tilde P_1 + m_8 \tilde P_2 + m_{10}^{(1)} \tilde P_3 \Big)$$ from $F$ to obtain an even-weight form with (at least) a fourth-order zero on $\mathcal{H}_1$ and (at least) a double zero on $\mathcal{H}_3$.

Suppose still that $F$ has order exactly four on $\mathcal{H}_1$. Then the quasi-pullback of $F$ to $\mathcal{H}_1$ is a Siegel modular form of even weight with at least an fourth-order zero on the diagonal (due to the double zero of $F$ on $\mathcal{H}_3$) and is therefore a multiple of $\psi_{10}^2$. Since $b_8^2$ has a fourth-order zero on $\mathcal{H}_1$ with quasi-pullback (up to scalar multiple) $\psi_{10}^2$, and it also has a double zero along $\mathcal{H}_3$, we may subtract away some expression of the form $b_8^2 P(\mathcal{E}_4,\mathcal{E}_6,m_{10}^{(2)},m_{12})$ from $F$ to obtain a modular form which vanishes to at least order $5$ along $\mathcal{H}_1$ and which has at least a double zero on $\mathcal{H}_3$.

Now if $F$ has order at least $5$ along $\mathcal{H}_1$ and a zero on $\mathcal{H}_3$, then the quotient $F / b_5$ is holomorphic (by Koecher's principle) and has lower weight, so $F/b_5$ and therefore $F$ is a polynomial expression in the generators in the claim.
\end{proof}

\begin{prop} The graded ring of Hermitian modular forms of degree 2 for $\mathbb{Q}(\sqrt{-11})$ is generated by the symmetric generators of Theorem 10 and the holomorphic quotients $$h_{24+2N} = \frac{b_{24} m_7^N}{b_5^N}, \; 0 \le N \le 5$$ and $$h_{24 + 2N+ 3} = \frac{b_{24} b_8 m_7^N}{b_5^{N+1}}, \; 0 \le N \le 4.$$
\end{prop}
\begin{proof} In the even-weight case our goal is to reduce against the skew-symmetric Borcherds product $b_{24}$ with divisor $$\mathrm{div} \, b_{24} = 11 \mathcal{H}_1 + \mathcal{H}_{11}.$$ To show that the pullbacks to $\mathcal{H}_1$ of odd orders $1 \le N \le 9$ are surjective it is enough to find skew-symmetric modular forms of weights $35-N$ with exactly an $N^{\mathrm{th}}$ order zero on $\mathcal{H}_1$ (whose $N^{\mathrm{th}}$ pullback must then be a multiple of $\psi_{35}$), since we have already produced preimages of the even-weight Siegel modular forms. It is easy to see that the quotients $h_{24+2N} = b_{24} (m_7 / b_5)^N$ are holomorphic and have order $11-2N$ on $\mathcal{H}_1$.

We will reduce odd-weight skew-symmetric forms $F$ to even-weight skew-symmetric forms by reducing against $b_5$. (The reduction against $b_{24}$ as in the previous paragraph fails as there are no skew-symmetric modular forms of weight 25.) First we force at least a fifth-order zero on $\mathcal{H}_1$ using the holomorphic forms $$h_{24 + 2N + 3} = \frac{b_{24} b_8 m_7^N}{b_5^{N+1}}, \; 2 \le N \le 4,$$ which have a zero of order $8-2N$ on $\mathcal{H}_1$ and whose quasi-pullbacks must be scalar multiples of $\psi_{35}$. Therefore by subtracting away expressions of the form $$\{h_{31},h_{33},h_{35}\} \cdot P(\mathcal{E}_4,\mathcal{E}_6,m_{10}^{(2)},m_{12})$$ we may assume that $F$ has at least a sixth-order zero on $\mathcal{H}_5$.

Now the pullback of $F$ to $\mathcal{H}_3$ is an skew-symmetric modular form of odd weight with at least a sixth-order zero on the diagonal and is therefore contined in the ideal generated by $\varphi_6 \varphi_9 f_{12}$ and $\varphi_8 \varphi_9 f_{12}$. Up to scalar multiple these are exactly the pullbacks of $h_{27} = \frac{b_{24}b_8}{b_5}$ and $h_{29} = \frac{b_{24} b_8 m_7}{b_5^2}$ to $\mathcal{H}_3$. Since $h_{27}$ and $h_{29}$ both vanish to order at least $5$ on $\mathcal{H}_1$, we subtract away some expression $$h_{27} P_1(\mathcal{E}_4,\mathcal{E}_6,m_{10}^{(2)},m_{12}) + h_{29} P_2(\mathcal{E}_4,\mathcal{E}_6,m_{10}^{(2)},m_{12})$$ from $F$ to obtain a form (again called $F$) whose divisor contains $5 \mathcal{H}_1 + \mathcal{H}_3$ and which is therefore divisible by $b_5$. The quotient $F / b_5$ is skew-symmetric of even weight so the previous case applies.
\end{proof}

\section{Dimension formulas}

The task of computing ideals of relations is much easier if dimension formulas for the spaces of modular forms are available (for one thing, such formulas make it clear when enough relations have been found to generate the ideal). In principle the dimensions can always be calculated via a trace formula or Riemann-Roch theorem; however this is a rather lengthy computation which does not seem to appear explicitly in the literature. In this section we observe that those dimensions can be read off almost immediately from the method of proof in sections 4 and 5 above.

Recall that the Hilbert series of a finitely generated graded $\mathbb{C}$-algebra $M = \bigoplus_{k=0}^{\infty} M_k$ is $$\mathrm{Hilb}\, M = \sum_{k=0}^{\infty} (  \mathrm{dim}\, M_k) t^k \in \mathbb{Z}[|t|].$$

\subsection{Dimension formulas for $K = \mathbb{Q}(\sqrt{-7})$.} We will express the Hilbert series of dimensions of Hermitian modular forms for $\Gamma_K = \mathrm{SU}_{2,2}(\mathcal{O}_K)$ in terms of the Hilbert series for $\mathrm{Sp}_4(\mathbb{Z})$ and the symmetric paramodular group $K(2)^+ = \langle K(2), V_2 \rangle$ of level 2. Recall that the latter series are $$\sum_{k=0}^{\infty} \mathrm{dim}\, M_k(\mathrm{Sp}_4(\mathbb{Z})) t^k = \frac{1 + t^{35}}{(1 - t^4)(1-t^6)(1-t^{10})(1-t^{12})}$$ and $$\sum_{k=0}^{\infty} \mathrm{dim}\, M_k^{sym}(K(2)) t^k = \frac{(1 + t^{10})(1 + t^{11})}{(1 - t^4)(1-t^6)(1-t^8)(1-t^{12})}$$ corresponding to the ring decompositions $$M_{\ast}(\mathrm{Sp}_4(\mathbb{Z})) = \mathbb{C}[E_4,E_6,\psi_{10},\psi_{12}] \oplus \psi_{35} \mathbb{C}[E_4,E_6,\psi_{10},\psi_{12}]$$ and $$M_{2\ast}^{sym}(K(2)) = \mathbb{C}[E_4,E_6,\phi_8,\phi_{12}] \oplus \phi_{10} \mathbb{C}[E_4,E_6,\phi_8,\phi_{12}], \; \; M_{\ast}^{sym}(K(2)) = M_{2\ast}^{sym}(K(2)) \oplus \phi_{11} M_{2\ast-11}^{sym}(K(2)).$$ 

We first consider (graded-) symmetric even weight Hermitian modular forms. Write $$H_{even}(t) = \sum_{k \, \text{even}} \mathrm{dim}\, M_k^{sym}(\Gamma_K) t^k, \; \; H_{odd}(t) = \sum_{k \, \text{odd}} \mathrm{dim}\, M_k^{sym}(\Gamma_K) t^k.$$

Although we reduce against the product $b_7$ whose zero on the Heegner divisor $\mathcal{H}_2$ is simple, the proof of Theorem 7 suggests that we consider both the zeroth and first order pullbacks there; so altogether we take the tuple of pullbacks $$P = (P_0^{\mathcal{H}_1},P_2^{\mathcal{H}_1},P_0^{\mathcal{H}_2},P_1^{\mathcal{H}_2}) : M_{2 \ast}^{sym}(\Gamma_K) \longrightarrow M_{2 \ast}(\mathrm{Sp}_4(\mathbb{Z})) \oplus S_{2\ast+2}(\mathrm{Sp}_4(\mathbb{Z})) \oplus M_{2\ast}^{sym}(K(2)) \oplus S_{2\ast+1}^{sym}(K(2)).$$ Then we obtain the exact sequences $$0 \longrightarrow \mathrm{ker}\Big(P_0^{\mathcal{H}_2} : M_{2 \ast -7}^{sym}(\Gamma_K) \rightarrow M_{2 \ast-7}^{sym}(K(2))\Big) \stackrel{\times b_7}{\longrightarrow} M_{2 \ast}^{sym}(\Gamma_K) \stackrel{P}{\longrightarrow} \mathrm{im}\, P \longrightarrow 0$$ and $$0 \longrightarrow \psi_{10}^2 \cdot \Big( M_{2 \ast-20}(\mathrm{Sp}_4(\mathbb{Z})) \oplus M_{2 \ast - 18}(\mathrm{Sp}_4(\mathbb{Z})) \Big) \longrightarrow \mathrm{im}\, P \longrightarrow  M_{2 \ast}^{sym}(K(2)) \oplus M_{2 \ast+1}^{sym}(K(2)) \longrightarrow 0,$$ from which we obtain the Hilbert series $$\mathrm{Hilb}\, \mathrm{im}\, P = \frac{t^{18} + t^{20}}{(1 - t^4)(1-t^6)(1-t^{10})(1-t^{12})} + \frac{(1 + t^{10})^2}{(1-t^4)(1-t^6)(1-t^8)(1-t^{12})}$$ and \begin{align*} H_{even}(t) &= \mathrm{Hilb}\, \mathrm{im}\, P + t^7 \Big(H_{odd}(t) - \frac{(1+t^{10})t^{11}}{(1-t^4)(1-t^6)(1-t^8)(1-t^{12})} \Big) \\ &= t^7 H_{odd}(t) + \frac{1 + t^{10} - t^{26} - t^{28} - t^{30} + t^{38}}{(1-t^4)(1-t^6)(1-t^8)(1-t^{10})(1-t^{12})}. \end{align*} By reducing odd-weight symmetric forms against $b_7$ we obtain the exact sequences $$0 \longrightarrow M_{2\ast - 6}^{sym}(\Gamma_K) \stackrel{\times b_7}{\longrightarrow} M_{2\ast+1}^{sym}(\Gamma_K) \stackrel{P = (P_1^{\mathcal{H}_1},P_0^{\mathcal{H}_2})}{\longrightarrow} \mathrm{im}\, P \longrightarrow 0$$ and $$0 \longrightarrow \psi_{10} \cdot M_{2 \ast - 9}(\mathrm{Sp}_4(\mathbb{Z})) \longrightarrow \mathrm{im}\, P \longrightarrow M_{2\ast+1}^{sym}(K(2)) \longrightarrow 0$$ and therefore $$H_{\mathrm{odd}}(t) = t^7 H_{\mathrm{even}}(t) + \frac{t^9}{(1-t^4)(1-t^6)(1-t^{10})(1-t^{12})} + \frac{(1+t^{10})t^{11}}{(1 - t^4)(1-t^6)(1-t^8)(1-t^{12})}.$$ These equations resolve to \begin{align*} &\quad \mathrm{Hilb}\, M_{\ast}^{sym}(\Gamma_K) = H_{even}(t) + H_{odd}(t)  \\ &= \frac{1 + t^4 + t^8 + t^9 + t^{10} + t^{11} + t^{12} + t^{13} + t^{14} + t^{15} + t^{16} + t^{18} + t^{19} + t^{20} + t^{22} + t^{23} + t^{24} + t^{27} - t^{30} - t^{34}}{(1 - t^6)(1-t^7)(1-t^8)(1-t^{10})(1-t^{12})}. \end{align*}

Now we compute dimensions of spaces of (graded) skew-symmetric modular forms. For even-weight forms the first, third and fifth order pullbacks to $\mathcal{H}_1$ yield an exact sequence $$0 \longrightarrow M_{2\ast-28}^{sym}(\Gamma_K) \stackrel{\times b_{28}}{\longrightarrow} M_{2\ast}^{skew}(\Gamma_K) \stackrel{(P_1,P_3,P_5)}{\longrightarrow}  S_{2\ast+1}(\mathrm{Sp}_4(\mathbb{Z})) \oplus S_{2\ast+3}(\mathrm{Sp}_4(\mathbb{Z})) \oplus S_{2\ast+5}(\mathrm{Sp}_4(\mathbb{Z})) \longrightarrow 0$$ and we obtain the generating series $$\sum_{k=0}^{\infty} \mathrm{dim}\, M_{2k}^{skew}(\Gamma_K) t^{2k} = \frac{t^{30} + t^{32} + t^{34}}{(1 - t^4)(1-t^6)(1-t^{10})(1 - t^{12})} + t^{28} \sum_{k=0}^{\infty} \mathrm{dim}\, M_{2k}^{sym}(\Gamma_K) t^{2k}.$$

As for odd-weight skew-symmetric forms, we use the exact sequences $$0 \longrightarrow M_{2\ast-6}^{skew}(\Gamma_K) \stackrel{\times b_7}{\longrightarrow} M_{2\ast+1}^{skew}(\Gamma_K) \stackrel{P = (P_0^{\mathcal{H}_1},P_2^{\mathcal{H}_1},P_0^{\mathcal{H}_2})}{\longrightarrow} \mathrm{im}\, P \longrightarrow 0$$ and $$0 \longrightarrow \phi_8 \phi_{11} f_{12} M_{2\ast-30}^{sym}(K(2)) \longrightarrow  \mathrm{im}\, P \longrightarrow M_{2\ast+1}(\mathrm{Sp}_4(\mathbb{Z})) \oplus M_{2\ast+3}(\mathrm{Sp}_4(\mathbb{Z})) \longrightarrow 0$$ to obtain \begin{align*} \sum_{k=0}^{\infty} \mathrm{dim}\, M_{2k+1}^{skew}(\Gamma_K) t^{2k+1} &= \frac{t^{33} + t^{35}}{(1 - t^4)(1-t^6)(1-t^{10})(1 - t^{12})} + \frac{t^{31}(1 + t^{10})}{(1 - t^4)(1-t^6)(1-t^8)(1-t^{12})} \\ &\quad \quad + t^7 \sum_{k=0}^{\infty} \mathrm{dim}\, M_{2k}^{skew}(\Gamma_K) t^{2k}, \end{align*} reducing the computation to the previous paragraph. Altogether we find $$\sum_{k=0}^{\infty} \mathrm{dim}\, M_k(\Gamma_K) t^k = \frac{P(t)}{(1-t^4)(1 - t^6)(1-t^7)(1-t^{10})(1-t^{12})}$$ where \begin{align*}P(t) &= 1 + t^{8} + t^{9} + t^{10} + t^{11} + t^{16} + t^{18} + t^{19} + t^{24}+ t^{27} + 2t^{32} + t^{33} + t^{34} + 2t^{35} - t^{42} + t^{43}. \end{align*}

The table below lists dimensions for the full space of Hermitian modular forms; the subspace of graded-symmetric Hermitian modular forms; and the subspace of Maass lifts.

\begin{table}[htbp]
\centering
\caption{Dimensions for $\mathbb{Q}(\sqrt{-7})$}
\begin{tabular}{l*{21}{c}r}
\hline
$k$ & 1 & 2 & 3 & 4 & 5 & 6 & 7 & 8 & 9 & 10 & 11 & 12 & 13 & 14 & 15 & 16 & 17 & 18 & 19 & 20 \\
\hline
$\mathrm{dim}\, M_k(\Gamma_K)$ & 0 & 0 & 0 & 1 & 0 & 1 & 1 & 2 & 1 & 3 & 2 & 4 & 2 & 5 & 4 & 8 & 5 & 10 & 8 & 13 \\
\hline
$\mathrm{dim}\, M_k^{sym}(\Gamma_K)$ & 0 & 0 & 0 & 1 & 0 & 1 & 1 & 2 & 1 & 3 & 2 & 4 & 2 & 5 & 4 & 8 & 5 & 10 & 8 & 13 \\
\hline
$\mathrm{dim}\, \mathrm{Maass}_k(\Gamma_K)$ & 0 & 0 & 0 & 1 & 0 & 1 & 1 & 2 & 1 & 3 & 2 & 3 & 2 & 4 & 3 & 5 & 3 & 5 & 4 & 6  \\
\hline
\hline
$k$ & 21 & 22 & 23 & 24 & 25 & 26 & 27 & 28 & 29 & 30 & 31 & 32 & 33 & 34 & 35 & 36 & 37 & 38 & 39 & 40 \\
\hline
$\mathrm{dim}\, M_k(\Gamma_K)$ & 10 & 17 & 14 & 22 & 17 & 26 & 23 & 35 & 28 & 42 & 37 & 52 & 44 & 63 & 57 & 76 & 66 & 90 & 84 & 109 \\
\hline
$\mathrm{dim}\, M_k^{sym}(\Gamma_K)$ & 10 & 17 & 14 & 22 & 17 & 26 & 23 & 34 & 28 & 41 & 36 & 50 & 43 & 60 & 54 & 72 & 63 & 84 & 78 & 101 \\
\hline
$\mathrm{dim}\, \mathrm{Maass}_k(\Gamma_K)$ & 4 & 7 & 5 & 7 & 5 & 8 & 6 & 9 & 6 & 9 & 7 & 10 & 7 & 11 & 8 & 11 & 8 & 12 & 9 & 13 \\
\hline
\end{tabular}    
\end{table}

%As for odd-weight modular forms, we have the surjective map $$(P_0,P_2) : M^{skew}_{2k+1}(\Gamma_K) \rightarrow S_{2k+1}(\mathrm{Sp}_4(\mathbb{Z})) \times S_{2k+3}(\mathrm{Sp}_4(\mathbb{Z}))$$ and $$\sum_{k=0}^{\infty} \mathrm{dim}\, M_{2k+1}^{skew}(\Gamma_K) t^{2k+1} - \frac{t^{33} + t^{35}}{(1 - t^4)(1-t^6)(1-t^{10})(1 - t^{12})}$$ is a generating series for spaces of antisymmetric modular forms which vanish to at least order $4$ on the Siegel upper half-space. From such a modular form one can subtract away some expression $h_{31} P(\mathcal{E}_4,\mathcal{E}_6,m_8,m_{10}^{(1)},m_{12})$ (corresponding to a symmetric paramodular form of weight $(2k+1) - 31$) to obtain a multiple of $b_7$. Therefore \begin{align*} \sum_{k=0}^{\infty} \mathrm{dim}\, M_{2k+1}^{skew}(\Gamma_K) t^{2k+1} &= \frac{t^{33} + t^{35}}{(1 - t^4)(1-t^6)(1-t^{10})(1 - t^{12})} + \frac{t^{31}(1 + t^{10})}{(1 - t^4)(1-t^6)(1-t^8)(1-t^{12})} \\ &\quad \quad + t^7 \sum_{k=0}^{\infty} \mathrm{dim}\, M_{2k}^{skew}(\Gamma_K) t^{2k}, \end{align*} reducing the computation to the previous paragraph. Altogether we find $$\sum_{k=0}^{\infty} \mathrm{dim}\, M_k(\Gamma_K) t^k = \frac{P(t)}{(1-t^4)(1 - t^6)(1-t^7)(1-t^{10})(1-t^{12})}$$ where \begin{align*}P(t) &= 1 + t^{8} + t^{9} + t^{10} + t^{11} + t^{16} + t^{18} + t^{19} + t^{24}+ t^{27} + 2t^{32} + t^{33} + t^{34} + 2t^{35} - t^{42} + t^{43}. \end{align*}

\subsection{Dimension formulas for $K = \mathbb{Q}(\sqrt{-11})$.} The procedure we use to compute Hilbert series of Hermitian modular forms for the field $\mathbb{Q}(\sqrt{-11})$ is mostly the same as the previous subsection. Here we need the corresponding series for symmetric paramodular forms of level three: $$\sum_{k=0}^{\infty} \mathrm{dim}\, M_k^{sym}(K(3)) t^k = \frac{1 + t^8 + t^{9} + t^{10} + t^{11} + t^{19}}{(1 - t^4)(1-t^6)^2 (1 - t^{12})}.$$ (This can be derived from Corollary 5.6 of \cite{Dern2} or computed directly. We remark that the series presented in \cite{Dern2} do not agree with this because the definition of ``symmetric" there is not graded-symmetric.) \\

Again write $$H_{even}(t) = \sum_{k \, \text{even}} \mathrm{dim}\, M_k^{sym}(\Gamma_K) t^k, \; \; H_{odd}(t) = \sum_{k \, \text{odd}} \mathrm{dim}\, M_k^{sym}(\Gamma_K) t^k.$$ Let $P =  (P_0^{\mathcal{H}_1},P_2^{\mathcal{H}_1},P_4^{\mathcal{H}_1},P_0^{\mathcal{H}_3},P_1^{\mathcal{H}_3}) $ denote the tuple of pullbacks $$P : M_{2*}^{sym}(\Gamma_K) \rightarrow M_{2*}(\mathrm{Sp}_4(\mathbb{Z})) \oplus S_{2*+2}(\mathrm{Sp}_4(\mathbb{Z})) \oplus S_{2*+4}(\mathrm{Sp}_4(\mathbb{Z})) \oplus M_{2*}^{sym}(K(3)) \oplus S_{2*+1}^{sym}(K(3)).$$ Reducing graded-symmetric even-weight forms against $b_5$ yields the exact sequences $$0 \rightarrow \mathrm{ker}\Big(P_0^{\mathcal{H}_3} : M_{2\ast-5}^{sym}(\Gamma_K) \rightarrow M_{2\ast-5}^{sym}(K(3)) \Big) \stackrel{\times b_5}{\longrightarrow} M_{2 \ast}^{sym}(\Gamma_K) \stackrel{P}{\longrightarrow} \mathrm{im}\, P \rightarrow 0,$$ $$0 \rightarrow \psi_{10}^2 \cdot \Big( \bigoplus_{k \in \{0,2,4\}} M_{2\ast - 20 + 2k}(\mathrm{Sp}_4(\mathbb{Z})) \Big) \longrightarrow \mathrm{im}\, P \longrightarrow M_{2 \ast}^{sym}(K(3)) \oplus M_{2\ast+1}^{sym}(K(3)) \longrightarrow 0,$$ from which we obtain $$\mathrm{Hilb}\, \mathrm{im}\, P = \frac{t^{16} + t^{18} + t^{20}}{(1-t^4)(1-t^6)(1-t^{10})(1-t^{12})} + \frac{1 + 2t^8 + 2t^{10} + t^{18}}{(1-t^4)(1-t^6)^2(1-t^{12})}$$ and \begin{align*} H_{even}(t) &= \mathrm{Hilb} \, \mathrm{im}\, P + t^5 \Big( H_{odd}(t) - \frac{t^9 + t^{11} + t^{19}}{(1-t^4)(1-t^6)^2(1-t^{12})} \Big) \\ &= t^5 H_{odd}(t) + \frac{1 + 2t^8 + t^{10} - t^{14} - t^{20} - t^{22} - t^{24} - t^{28} + t^{34}}{(1-t^4)(1-t^6)^2(1-t^{10})(1-t^{12})}. \end{align*}

Similarly, the reduction of odd-weight symmetric forms against $b_5$ through the tuple of pullbacks $P = (P_1^{\mathcal{H}_1},P_3^{\mathcal{H}_1},P_0^{\mathcal{H}_3})$ yields the exact sequences $$0 \longrightarrow M_{2\ast-4}^{sym}(\Gamma_K) \stackrel{\times b_5}{\longrightarrow} M_{2\ast+1}^{sym}(\Gamma_K) \stackrel{P}{\longrightarrow} \mathrm{im}\, P \longrightarrow 0$$ and $$0 \longrightarrow \psi_{10} \cdot \Big( M_{2\ast - 9}(\mathrm{Sp}_4(\mathbb{Z})) \oplus M_{2\ast - 7}(\mathrm{Sp}_4(\mathbb{Z})) \Big) \longrightarrow \mathrm{im}\, P \longrightarrow M_{2\ast+1}^{sym}(K(3)) \longrightarrow 0,$$ so $$H_{odd}(t) = \frac{t^7 + t^9}{(1-t^4)(1-t^6)(1-t^{10})(1-t^{12})} + \frac{t^9 + t^{11} + t^{19}}{(1-t^4)(1-t^6)^2 (1-t^{12})} + t^5 H_{even}(t).$$ Altogether we find \begin{align*} \mathrm{Hilb} \, M_*^{sym}(\Gamma_K) &= H_{even}(t) + H_{odd}(t) \\ &= \frac{1 + t^5 + t^7 + 2t^8 + 2t^9 + 2t^{10} + t^{11} + t^{12} + t^{13} + t^{14} + t^{15} + t^{16} + t^{17} + t^{18} + t^{19} + t^{23} - t^{29}}{(1-t^4)(1-t^6)^2(1-t^{10})(1-t^{12})}. \end{align*}

For skew-symmetric modular forms we argue as in the previous subsection and find $$\sum_{k=0}^{\infty} \mathrm{dim}\, M_{2k}^{skew}(\Gamma_K) t^{2k} = \frac{t^{26} + t^{28} + t^{30} + t^{32} + t^{34}}{(1 - t^4)(1-t^6)(1-t^{10})(1-t^{12})} + t^{24} \sum_{k=0}^{\infty} \mathrm{dim}\, M_{2k}^{sym}(\Gamma_K) t^{2k}$$ and \begin{align*} \sum_{k=0}^{\infty} \mathrm{dim}\, M_{2k+1}^{skew}(\Gamma_K) t^{2k+1} &= \frac{t^{31} + t^{33} + t^{35}}{(1 - t^4)(1-t^6)(1-t^{10})(1-t^{12})} + \frac{t^{27} + t^{29} + t^{37}}{(1 - t^4)(1-t^6)^2 (1 - t^{12})} \\ &\quad + t^5 \sum_{k=0}^{\infty} \mathrm{dim}\, M_{2k}^{skew}(\Gamma_K) t^{2k}, \end{align*} and altogether $$ \sum_{k=0}^{\infty} \mathrm{dim}\, M_k(\Gamma_K) t^k = \frac{P(t)}{(1-t^4)(1-t^5)(1-t^6)^2(1-t^{12})}$$ where \begin{align*} P(t) &= 1 + t^{7} + 2t^{8} + 2t^{9} + 2t^{10} + t^{11} - t^{13} - t^{14} - t^{15} + t^{17} + 2t^{18} + 2t^{19} + t^{20} \\ &- t^{22} - t^{23} - t^{24} - t^{25} + t^{26} + 2t^{27} + 2t^{28} + 2t^{29} + 2t^{30} + t^{31} - t^{36} + t^{37}. \end{align*}

The table below lists dimensions for the full space of Hermitian modular forms; the subspace of graded-symmetric Hermitian modular forms; and the subspace of Maass lifts.
\begin{footnotesize}
\begin{table}[htbp]
\centering
\caption{Dimensions for $\mathbb{Q}(\sqrt{-11})$}
\begin{tabular}{l*{21}{c}r}
\hline
$k$ & 1 & 2 & 3 & 4 & 5 & 6 & 7 & 8 & 9 & 10 & 11 & 12 & 13 & 14 & 15 & 16 & 17 & 18 & 19 & 20 \\
\hline
$\mathrm{dim}\, M_k(\Gamma_K)$ & 0 & 0 & 0 & 1 & 1 & 2 & 1 & 3 & 3 & 5 & 4 & 8 & 6 & 10 & 10 & 15 & 14 & 21 & 19 & 28  \\
\hline
$\mathrm{dim}\, M_k^{sym}(\Gamma_K)$ & 0 & 0 & 0 & 1 & 1 & 2 & 1 & 3 & 3 & 5 & 4 & 8 & 6 & 10 & 10 & 15 & 14 & 21 & 19 & 28  \\
\hline
$\mathrm{dim}\, \mathrm{Maass}_k(\Gamma_K)$ & 0 & 0 & 0 & 1 & 1 & 2 & 1 & 3 & 3 & 4 & 3 & 5 & 4 & 6 & 5 & 7 & 6 & 8 & 6 & 9   \\
\hline
\hline
$k$ & 21 & 22 & 23 & 24 & 25 & 26 & 27 & 28 & 29 & 30 & 31 & 32 & 33 & 34 & 35 & 36 & 37 & 38 & 39 & 40 \\
\hline
$\mathrm{dim}\, M_k(\Gamma_K)$ & 27 & 36 & 35 & 49 & 45 & 60 & 60 & 77 & 76 & 98 & 94 & 120 & 120 & 147 & 147 & 181 & 177 & 216 & 219 & 260 \\
\hline
$\mathrm{dim}\, M_k^{sym}(\Gamma_K)$ &27 & 36 & 35 & 48 & 45 & 59 & 59 & 75 & 74 & 94 & 91 & 114 & 114 & 138 & 138 & 168 & 165 & 198 & 200 & 236 \\
\hline
$\mathrm{dim}\, \mathrm{Maass}_k(\Gamma_K)$ & 8 & 10 & 8 & 11 & 9 & 12 & 10 & 13 & 11 & 14 & 11 & 15 & 13 & 16 & 13 & 17 & 14 & 18 & 15 & 19 \\
\hline
\end{tabular}    
\end{table}
\end{footnotesize}

\bibliographystyle{plainnat}
\bibliofont
\bibliography{\jobname}
\addresseshere

\end{document}